\documentclass{article}
\usepackage{tcolorbox}
\usepackage[utf8]{inputenc}
\usepackage[english]{babel}
\usepackage[a4paper, total={6in, 8in}]{geometry}

\usepackage{indentfirst,enumerate,amsthm}
\usepackage[inline]{enumitem}
\usepackage{amssymb,amsmath,mathrsfs,dsfont,nicefrac}
\usepackage{chapterbib}
\usepackage{graphicx,subfig}
\usepackage{algorithm2e}
\RestyleAlgo{ruled}
\usepackage{appendix}
\usepackage{xcolor}
\usepackage{hyperref}
\usepackage[lowercase]{theoremref}
\hypersetup{
    colorlinks=true,
    linkcolor=blue,
    filecolor=magenta,      
    urlcolor=cyan,
    pdftitle={Overleaf Example},
    pdfpagemode=FullScreen,
    }
    
\DeclareMathOperator*{\argmin}{argmin}

\DeclareMathOperator{\ksd}{KSD}
\DeclareMathOperator{\wksd}{wKSD}
\DeclareMathOperator{\tr}{tr}
\DeclareMathOperator{\F}{F}
\DeclareMathOperator{\SO}{SO}
\DeclareMathOperator{\GL}{GL}
\DeclareMathOperator{\Or}{O}

\DeclareMathOperator{\Div}{div}

\DeclareMathOperator{\Log}{Log}
\DeclareMathOperator{\vectz}{vec}

\DeclareMathOperator{\KL}{KL}

\theoremstyle{plain}
\newtheorem*{theorem*}{Theorem}
\newtheorem{theorem}{Theorem}[section]

\newtheorem{proposition}{Proposition}[section]
\newtheorem{lemma}{Lemma}[section]

\theoremstyle{remark}
\newtheorem{definition}{Definition}[section]
\newtheorem*{remark}{Remark}
\newtheorem{example}{Example}[section]

\newtheorem{assumption}{Assumption}
\tcolorboxenvironment{assumption}{colback = gray!5!white, colframe = gray!75!black}

\title{Theory and Applications of Kernel Stein Discrepancy on Riemannian Manifolds
}

\author{Xiaoda Qu$^\dagger$ and Baba C. Vemuri$^{\ddag}$
\\$^{\dagger}$Department of Statistics \ \ $^{\ddag}$Department of CISE\\University of Florida}

\begin{document}

\maketitle

\begin{abstract}

Distributional comparison is a fundamental problem in statistical data analysis with numerous applications in a variety of scientific and engineering fields. Numerous methods exist for distributional comparison but kernel Stein's method has gained significant popularity in recent times. In this paper, we first present a novel mathematically rigorous and consistent generalization of the Stein operator to Riemannian manifolds. Then we show that the kernel Stein discrepancy (KSD) defined via this operator is nearly as strong as the KSD in the Euclidean setting in terms of  distinguishing the target distributions from the reference. We investigate the asymptotic properties of the minimum kernel Stein discrepancy estimator (MKSDE), apply it to goodness-of-fit testing, and compare it to the maximum likelihood estimator (MLE) experimentally. We present several examples of our theory applied to commonly encountered Riemannian manifolds in practice namely, the n-sphere, the Grassmann, Stiefel, the manifold of symmetric positive definite matrices and other Riemannian homogeneous spaces. On the aforementioned manifolds, we consider a variety of distributions with intractable normalization constants and derive closed form expressions for the KSD and MKSDE. 
\end{abstract}

\section{Introduction}

Data residing in curved spaces have recently received growing attention in numerous fields of Science and Engineering. To model their underlying curved geometry, it is natural to model the space in which they reside with known manifold geometries for example, (i) the Stiefel manifold, $\mathcal{V}_r(N)$, commonly used to model the space of directional data in field of computer vision  \cite{chakraborty2019statistics,turaga2008statistical}, 
dynamic systems \cite{BRIDGES2001219} and rigid body motion \cite{ley2018applied,oualkacha2012estimation,schulz2015analysis},
(ii) Grassmann manifold $\mathcal{G}_r(N)$ in signal processing \cite{dai2008quantization,mondal2007quantization,AnujS-TSP200}, shape analysis \cite{goodall1999projective,Bauer-JMIV2014,yataka2023grassmann} and image processing \cite{dong2017classification,dong2013clustering,sharma2020image}, 
(iii) covariance matrices are modeled as points on
the manifold of symmetric positive definite (SPD) matrices $\mathcal{P}(N)$ in diffusion magnetic resonance imaging \cite{basser1994mr,Fletcher-Joshi2007,Wang-Vemuri2005TMI,le2001diffusion} and brain computer interfaces \cite{carrara2024geometric,xu2023signature}.
  

However, due to the lack of vector space structure, significant complications arise in the formulation and application of statistical methods to such curved spaces. e.g., the data points lying on manifolds can not be simply summed up, thus the notion of the classical arithmetic mean is not meaningful in general. Among all the challenges, the most significant one is the issue of normalization constant associated with probability distributions defined on the manifold-valued  ($M$-valued) random variables, which arises in distributional comparison, parametric estimation and numerous practical applications. Even the simplest distribution on the simplest curved manifold, e.g., the von Mises Fisher distribution $p(x)\propto \exp(\mu^\top x)$ on a sphere $\mathbb{S}^{d-1}$, has a normalization constant that is intractable. Furthermore, the KL-divergence KL$(p,q):=\mathbb{E}_p[\log\frac{p}{q}]$, which is the most commonly-used loss function in parameter estimation, distributional comparison and neural network training, highly relies on the computation of normalization constant of $p$ and $q$. 

In practice, approximating these constants and their derivative with respect to 
the parameters of the distribution requires the use of numerical methods such as the gradient descent and/or its variants, which has been resorted to by many researchers in statistics, machine learning and robotics literature \cite{gilitschenski2020deep,glover2014tracking, kume2013saddlepoint, kume2018exact}
but at the expense of a high computational cost. It would of course be more desirable to fully avoid computing this intractable constant and simultaneously achieve high accuracy in parameter estimation. In fact, this will be the {\it main objective of this paper}.

\emph{Kernel Stein Discrepancy} (KSD), a normalization free loss function was first introduced by Liu et al. \cite{liu2016kernelized} as a measure of goodness of fit and for model evaluation. KSD measures the difference between distributions by using a combination of the so called Stein's method and the well established reproducing kernels Hilbert space (RKHS) theory.
 KSD has since been extensively researched on, including various aspects within a general framework \cite{ley2017stein, mijoule2018stein, mijoule2021stein,oates2017control}, its characterization scope \cite{gorham2015measuring, gorham2017measuring,qu2022framework,barp2024targeted}, exploration of its asymptotic properties relating to minimization \cite{barp2019minimum, oates2022minimum}, its diverse applications \cite{chwialkowski2016kernel, liu2016kernelized, matsubara2022robust}, optimality of KSD-based goodness-of-fit test \cite{hagrass2025minimaxoptimalgoodnessoffittesting} and generalizations to manifolds \cite{barp2018riemann,xu2021interpretable,le2024diffusion,Qu2024kernel}.
 
\subsection{Context}

At its core, the KSD consists of a RKHS, $\mathcal{H}_\kappa$, defined by a kernel function $\kappa$, known as \emph{Stein's class}. Furthermore, it incorporates a \emph{Stein's operator} $\mathcal{S}_P$, dependent on the candidate distribution $P$, but independent of its normalizing constant. The role of the Stein operator is to map elements from $\mathcal{H}_\kappa$ to real integrable functions. These operators must satisfy \emph{Stein's identity}, given by, $P(\mathcal{S}_P f)=0$ (the integral of $\mathcal{S}_P f$ w.r.t. $P$) for all $f\in \mathcal{H}_\kappa$. This leads us to the Stein pair defined as:
\begin{tcolorbox}[colback=gray!5!white,colframe=gray!75!black,title=Stein pair]
A pair $(\mathcal{S}_P,\mathcal{H}_\kappa)$ consisting of a \emph{Stein operator} $\mathcal{S}_P$ and a \emph{Stein class} $\mathcal{H}_\kappa$ satisfying \emph{Stein's identity}, i.e.,  $P(\mathcal{S}_P f)=0$ for all $f\in\mathcal{H}_\kappa$, is called a \emph{Stein pair}.
\end{tcolorbox}

With these foundational elements in place, the KSD between distributions $P$ and $Q$ can be defined as follows:
\begin{equation}\label{KSDdef}
    \ksd(P,Q):= \sup\big\{  Q(\mathcal{S}_p f): f\in\mathcal{H}_\kappa, \Vert f\Vert_{\mathcal{H}_\kappa}\leq 1 \big\}.
\end{equation}


\subsubsection{KSD on \texorpdfstring{\(\mathbb{R}^d\)}{R\^d}} \label{KSDE} Most of the existing works have focused on the KSD defined on $\mathbb{R}^d$. Let $\mathcal{H}_\kappa$ be a RKHS on $\mathbb{R}^d$ associated with kernel $\kappa$, and $\mathcal{H}^d_\kappa$ be the $d$-fold Cartesian product of $\mathcal{H}_\kappa$, equipped with the inner product
$ \langle \Vec{f},\Vec{g}\rangle_{\mathcal{H}_\kappa^d}=\sum_{l=1}^d\langle f_l,g_l \rangle_{\mathcal{H}_\kappa} $ for $\Vec{f}=(f_1,\dots,f_d) $ and $\Vec{g}$ in $\mathcal{H}^d_\kappa$. Suppose $P$ has a differentiable density $p$ w.r.t. the Lebesgue measure. The the most commonly adopted Stein operator $\mathcal{A}_p$ on $ \mathbb{R}^d$ \cite{barp2019minimum,chwialkowski2016kernel,gorham2017measuring,liu2016kernelized,matsubara2022robust,oates2022minimum} is defined as
\begin{equation}\label{StOE}
    \mathcal{A}_p: \Vec{f} \mapsto \sum_{l=1}^d \left[\frac{\partial f_l}{\partial x^l} + f^l \frac{\partial}{\partial x^l} \log p\right],\quad \Vec{f}\in \mathcal{H}^d_\kappa.
\end{equation}
KSD is then obtained by substituting  (\ref{StOE}) into (\ref{KSDdef}), i.e., 
\begin{equation}\label{KSDdef-Rn}
    \ksd(P,Q):= \sup\big\{  Q(\mathcal{A}_p \Vec{f}): \Vec{f}\in\mathcal{H}^d_\kappa, \Vert \Vec{f}\Vert_{\mathcal{H}^d_\kappa}\leq 1 \big\}.
\end{equation}
Clearly, one can easily see from the definition and Stein's identity that $\ksd(P,Q)\geq 0$ and $\ksd(P,P)=0$. In fact, as demonstrated in \cite{barp2019minimum,chwialkowski2016kernel,gorham2017measuring,liu2016kernelized,matsubara2022robust,oates2022minimum}, if the kernel $\kappa$ is $C_0$-universal, then KSD separates (discriminates)  $P$ from $Q$ with $C^1$ densities, i.e.,  $\ksd(P,Q)=0 \Longrightarrow Q=P$. Furthermore, \cite{barp2024targeted} showed that if $\kappa$ is $C_0$-universal and translation-invariant, then $\ksd$ separates $P$ from all $Q$ (with or without a density). Notably, the computation of $\mathcal{A}_p$ is independent of the normalizing constant of $P$, so is the associated KSD in (\ref{KSDdef-Rn}).

The \emph{minimum kernel Stein discrepancy estimator} (MKSDE) is one of the most important applications of KSD, which was first proposed in \cite{barp2019minimum} and further investigated in \cite{matsubara2022robust,oates2022minimum}. The MKSDE minimizes the KSD between the empirical distribution and a parametrized family $p_\theta$, to acquire an estimate \(p_{\theta^*}\) of the underlying distribution of the samples. The MKSDE converges under mild regularity conditions and thus can serve as a normalization-free alternative to MLE.

\subsubsection{ Existing generalizations to manifold} In contrast to the extensive work on KSD in \(\mathbb{R}^d\), generalizations of KSD to Riemannian manifolds is scarce and the existing generalizations are somewhat restrictive as will be evident from  the following discussion.

In
\cite{barp2018riemann}, Barp et al. adopted the Stein operator $\mathcal{L}_p:f\mapsto \Delta f+ g(\nabla\log p,\nabla f) $ and the Sobolev space as their RKHS on {\it compact manifolds}. It is a significant challenge to identify a closed form kernel for such a large RKHS on a curved manifold. An alternative approach to tackle this challenge, not discussed in \cite{barp2018riemann}, is to restrict a Sobolev-type kernel from $\mathbb{R}^d$ to the manifold, as suggested in \cite[Thm. 5]{fuselier2012scattered}. However, this approach does not apply to most of the commonly-used kernels, e.g., Gaussian kernels, Inverse Multi-Quadric (IMQ) kernels, log-inverse kernels etc.

In \cite{le2024diffusion}, Le et al. also adopted the Stein operator $\mathcal{L}_p:f\mapsto \Delta f+g(\nabla \log p,\nabla f)$ and employed a diffusion-based Stein's method to obtain a bound on the 1-Wasserstein metric on complete Riemannian manifolds. Although the assumption made by this work, i.e., the Bakry-Emery curvature criterion $\text{Ric}-\text{Hess}\log p\geq 2\tau g$ for some $\tau>0$, is rather strong in practice and limits the practical applicability of this method, it nevertheless still generalized the classical diffusion approach in \cite{mackey2016multivariate} from $\mathbb{R}^d$ and broadened the methodological toolkit of Stein's method in the manifold settings.

In \cite{xu2021interpretable}, Xu et al. adopted the same Stein's operator \(\mathcal{A}_p\) in \eqref{StOE}, replacing the coordinate \(x^i\) with the local coordinate chart on the manifold, and applying Stokes's theorem to show Stein's identity. However, there is no global chart on any compact manifold, e.g., sphere $\mathbb{S}^{N-1}$, Stiefel manifold $\mathcal{V}_r(N)$, Grassmann manifold $\mathcal{G}_r(N)$ and many others.
In order for Stein's identity based on the Stokes's theorem to hold in their method, the target density $p$ must vanish outside the singular boundary of the chart on such compact manifolds. 

In \cite{Qu2024kernel}, Qu et al. also adopted \eqref{StOE}, but replaced the vector fields $\frac{\partial}{\partial x^l}$ in \eqref{StOE} with the left invariant vector fields utilizing the Lie groups structure, so as to circumvent the issue of local coordinates on Lie groups. However, there are many manifolds that are widely encountered in applications, including the sphere $\mathbb{S}^{N-1}$, Stiefel manifold $\mathcal{V}_r(N)$, Grassmann manifold $\mathcal{G}_r(N)$ and the manifold of symmetric positice definite matrices $\mathcal{P}(N)$ which do not possess a Lie group structure and will be addressed in this work.

\subsection{Our work and contributions} Our contributions in this work are itemized below. {\it The proofs of all theorems original to this manuscript are provided in appendix \ref{proofs_supp}.}
\begin{itemize}
\item {\it KSD on general Riemannian manifolds:} In \S\ref{Theory}, we propose a novel Stein operator on the general complete Riemannian manifold and study its properties. Unlike the previous works, it leads to a normalization-free KSD that is not only applicable to all complete Riemannian manifold but also independent of the choice of local coordinates. We also demonstrate in Thms. \ref{KSD-Characterization-Compact} and  \ref{KSD-Characterization-Noncompact} respectively that, our KSD achieves significantly stronger separation results at the expense of rather mild conditions being imposed on the kernel that are satisfied even by the most widely used kernels in practice. Compared to past works of \cite{barp2018riemann,xu2021interpretable}, our work noticeably expands the applicability of KSD in machine learning, engineering and other fields. In \S\ref{KSD-Homogeneous}, we show that the KSD can be further simplified on Riemannian homogeneous spaces utilizing isometry structure and killing vector fields (metric preserving vector fields). We will elaborate on these topics subsequently.

\item {\it MKSDE and its applications:} In \S\ref{MKSDE}, we introduce the MKSDE obtained by minimizing our novel KSD and its asymptotic properties. In \S\ref{GoF}, we introduce the composite goodness of fit test, one of the most important application based on MKSDE. These results follow the same outline in our previous work \cite{Qu2024kernel} on Lie groups, but we generalize the theory so that it is applicable to all complete Riemannian manifolds. In \S\ref{experiments}, we present two applications of our KSD to one of the most widely-encountered manifolds in science and engineering namely, the Stiefel manifold $\mathcal{V}_2(3)$. Specifically, the first experiments \ref{MKSDEvsMLE} will address the issue of the normalization constant that arises in MLE and how the estimation obtained using proposed normalization-free KSD yields far more accurate parameter estimates compared to MLE that uses approximations for the normalization constant. The second experient justifies the power of the composite goodness of fit test based on our MKSDE.

\item {\it Explicit closed forms:} The most significant property of our KSD and MKSDE, is that they have closed forms on some of the most widely-encountered manifolds, including Stiefel manifold $\mathcal{V}_r(N)$(including the sphere $\mathbb{S}^{N-1}$ and the rotation group $\text{SO}(N)$ as $\mathcal{V}_1(N)=\mathbb{S}^{N-1}$ and $\mathcal{V}_{N-1}(N)=\text{SO}(N)$), Grassmann manifold $\mathcal{G}_r(N)$ and the manifold of symmetric positive definite matrices $\mathcal{P}(N)$. In \S\ref{Examples}, we will compute the explicit form of our KSD and our MKSDE for the exponential family of distributions on these manifolds, which will  facilitate the usage of our method in practice.
\end{itemize}
\section{Mathematical Background}

In this section, we will introduce several pivotal theorems that will be used subsequently in this work. For the definitions of several relevant concepts used throughout this paper, we refer the readers to the following: differential geometry texts \cite{lee2006riemannian,lee2013smooth,petersen2016riemannian} for  the Riemannian manifold, the Riemannian metric, the Riemannian distance, the vector fields, local curves $[\mathfrak{c}(t)]$, volume measure, divergence operator, Riemannian gradient; \cite[\S A]{steinwart2008support}  for the notion of reproducing kernel Hilbert space (RKHS), kernel function and the Bochner integral.

\paragraph{Notation} For a Riemannian manifold $(M,g)$, we denote by $\rho$ the Riemannian distance function, by $\Omega$ the volume measure, by $\Delta$ the Laplacian operator, by $\Div$ the divergence operator, by $\nabla$ the Riemannian gradient operator. For a smooth vector field $D$ on $M$, let $|D|:=\sqrt{g(D,D)}$ be its pointwise length. The symbol $\mathcal{H}_\kappa$ will represent the reproducing kernel Hilbert space associated with the kernel function $\kappa$. Let $C(M)$ ($C_c(M)$, $C_b(M)$, $C_0(M)$ resp.) be the space of all (compactly supported, bounded, vanishing at infinity resp.) continuous functions on $M$, $C^k(M)$ be the space of all $k$-times continuously differentiable functions on $M$, $C^{(1,1)}(M)$ be the space of all $(1,1)$-times continuous differentiable bivariate function on $M\times M$. All measures considered in this work are defined on Borel algebra $\mathscr{B}(M)$. For measure $\nu$, let $\nu(f)$ denote the integral of $f$ w.r.t. $\nu$, and $L^2(\nu)$ denote the space of square integrable functions w.r.t. $\nu$. Denote by $\tau\ll\nu$ if another measure $\tau$ is absolutely continuous w.r.t. $\nu$, and $\frac{d\tau}{d\nu}$ denote the Radon-Nikodym derivative (R-N derivative). If $\tau\ll\nu$ and $\nu\ll\tau$, then we denote $\tau\sim\nu$. Let $\tau\times\nu$ be the product measure of $\tau$ and $\nu$. Let $\mathcal{P}(M)$ be the space of all probability measures on $M$. We say a sequence of $Q_n\in\mathcal{P}(M)$ converges to $Q\in\mathcal{P}(M)$ weakly if $Q_n(f)\to Q$ for all $f\in C_b(M)$, denoted by $Q_n\Rightarrow Q$.

The first three theorems will play fundamental roles in the construction of our Stein operator in \S\ref{StoM-sec}.
The first theorem \cite{gaffney1954special} generalizes the classical Stokes's theorem, which ensures the Stein's identity holds. 

\begin{theorem}[Divergence theorem] \label{DivThm} Let $M$ be a complete Riemannian manifold. For a locally Lipschitz continuous vector field $D$ on $M$ such that $|D|$ and $\Div D$ are both integrable w.r.t volume measure $\Omega$, the following identity holds: $\int_M \Div D d \Omega= 0.$
\end{theorem}

The second theorem \cite[\S A.5.4]{steinwart2008support} captures the interchangeability between the Bochner integral and a continuous linear functional, which serves as the key to obtaining the closed form of KSD in Thm. \ref{KSDFormThm}.
\paragraph{Bochner Integral} Suppose $P\in\mathcal{P}(M)$, $\mathcal{B}$ is a separable Banach space, and $\phi:M\to\mathcal{B}$ is a Borel measurable $\mathcal{B}$-valued map. A (measurable) step function is a map in the form of $\sum_{i=1}^m \textbf{1}_{A_i} x_i $ for some $x_1,\dots,x_n\in\mathcal{B}$ and $A_1,\dots,A_n\in\mathcal{F}$. For each measurable $\mathcal{B}$-valued map $\phi$, there exists a sequence of measurable step functions $\phi_n$ such that $\Vert \phi_n-\phi\Vert_{\mathcal{B}}\to 0 $ pointwisely. A measurable $\mathcal{B}$-valued map $\phi$ is \emph{Bochner $P$-integrable} if there exists a sequence of step functions $\phi_n=\sum_{i=1}^{m_n} \textbf{1}_{A_{i,n}}x_{i,n}$ such that $\lim_{n\to\infty} P(\Vert \phi_n-\phi\Vert)=0$, then the unique \emph{Bochner integral} of $\phi$ w.r.t. $P$ is defined as $P(\phi)=\lim_{n\to \infty} \sum_{i=1}^{m_n} P(A_{i,n}) x_{i,n} $. A measurable $\mathcal{B}$-valued map $\phi$ is $P$-integrable if and only if $P(\Vert \phi\Vert) <+\infty$. In addition, we have
\begin{theorem}\label{Bochner} Suppose $\phi$ is a $P$-integrable $\mathcal{B}$-valued map and $f$ is a continuous linear functional on $\mathcal{B}$, then $f[P(\phi)]=P[f(X)]$.
\end{theorem}

The third theorem, called Mercer's theorem \cite[\S 4.5]{steinwart2008support}, ensures the spectral decomposition of a kernel $\kappa$, which is pivotal in the depiction of the approximating distribution of the empirical KSDs in Thm. \ref{KSDasymptotic}.

\begin{theorem}[Mercer's theorem]\label{Mercer} Suppose $P\in\mathcal{P}(M)$ and $\kappa\in L^2(P\times P)$. There exists a sequence of positive numbers $\lambda_k$, $k\geq 1$ and a sequence of orthonormal eigen-functions $\phi_k\in L^2(P)$ such that, 
$
\sum_{k=1}^n \lambda_k \phi(x)\phi(y)\to \kappa(x,y) \text{\ \ in\ \ } L^2( P\times P), \text{\ as\ \ } n\to\infty.
$
Then, $\{\lambda_k\}$ are said to be the \emph{eigenvalues} of $\kappa$.
\end{theorem}

The fourth theorem \cite[Lem. 4.34]{steinwart2008support} demonstrates the connection between the differentiability of the kernel and the differentiability of the functions in its associated RKHS, which ensures that all the functions in the RKHS are differentiable so that the Stein operator is applicable.

\begin{theorem}
\label{kernel-differentiable}
If $\kappa\in C^{(1,1)}(M)$, then $\mathcal{H}_\kappa\subset C^1(M)$, and for a tangent vector $D\in T_{x_0} M$, we have $(D \kappa)_{x_0}\in \mathcal{H}_\kappa$ and $D f(x_0) = \langle f, (D\kappa)_{x_0}\rangle_{\mathcal{H}_\kappa}$. Here $(D \kappa)_{x_0}$ represents the function obtained by letting $D$ act on the first argument of $\kappa$ and fix the first argument at $x_0$.
\end{theorem}

Suppose a kernel $\kappa$ on $M$ is bounded, then $\varphi: P\mapsto \varphi_P(x):=\int \kappa(x,y) dP(y)\in\mathcal{H}_\kappa$ is a $\mathcal{H}_\kappa$-valued map from $\mathcal{P}(M)$ to $\mathcal{H}_\kappa$, namely, the \emph{kernel mean embedding} of $P$. A bounded kernel is said to be \emph{characteristic} if $\varphi$ is injective on $\mathcal{P}(M)$. Furthermore, $\kappa$ is said to be $C_0$-\emph{universal} if $\mathcal{H}_\kappa$ is a dense subspace of $C_0(M)$.

The characteristic kernels and $C_0$-universal kernels are two types of important kernels that satisfy our mathematical requirements and will be extensively used in \S\ref{KSD-Separability}. In practice, many commonly-used translation-invariant kernels on $\mathbb{R}^d$ fall into this category (e.g., Gaussian, Inverse Multi-Quadric (IMQ), Matérn, B-spline, Cauchy, sech, Wendland compact-support, spectral mixture and others). A kernel $\kappa$ on $\mathbb{R}^d$ is said to be \emph{translation-invariant} if $\kappa(x+z,y+z) = \kappa(x,y)$ for any $x,y,z\in\mathbb{R}^d$. Specifically, we will use following two most widely-used classes of kernels to explicitly calculate some examples in \S\ref{Examples}.

\begin{example}\label{Radial-Ex}
Following kernels will be used in \S \ref{Examples}:
\begin{itemize}
    \item Gaussian kernels: $
    \kappa(x,y)=\exp(-\frac{\tau}{2}\Vert x-y\Vert_{\mathbb{R}^d}^2)$ for some $\tau>0$.
    \item IMQ kernels: $
    \kappa(x,y)=(\beta+\Vert x-y\Vert_{\mathbb{R}^d}^2)^{-\gamma}$ for some $\beta,\gamma>0$.
\end{itemize}
These two classes of kernels are characteristic and $C_0$-universal. Notably, they all belong to the class of \emph{radial kernel}, i.e., there exists some $\psi\in C^2[0,+\infty)$ such that $\kappa(x,y)=e^{-\psi(\Vert x-y\Vert^2_{\mathbb{R}^d})}$. This property will significantly simplify the analytical derivation in \S\ref{Examples}.
\end{example}
\section{KSD on Riemannian Manifolds}\label{Theory}

\subsection{Stein operator on Riemannian Manifolds}\label{StoM-sec}

First we seek to generalize the Stein pair $(\mathcal{A}_p,\mathcal{H}^d_\kappa)$ on $\mathbb{R}^d$ defined in Eq. (\ref{StOE}) to Riemannian manifolds and then use it to develop the KSD. It is straightforward to see that following assumptions must be maintained on manifolds.

\begin{assumption}\label{assumption-M}
The Riemannian manifold $M$ is complete and connected. 
\end{assumption}

The completeness of the manifold ensures that the generalized divergence theorem (Thm. \ref{DivThm}) holds, by which we establish Stein's identity for our Stein pair in \S\ref{KSDonM}. The connectedness of the manifold will be used to establish the separation results of the KSD in \S\ref{KSD-Separability}.

\begin{assumption}
The target distribution $P$ has a locally Lipschitz continuous density $p>0$ w.r.t the volume measure $\Omega$.
\end{assumption}

The local Lipschitz continuity is to ensure that $p$ is almost everywhere differentiable by Rademacher theorem \cite[Thm. 3.1.6]{federer2014geometric}. Rather than assuming $p\in C^1(M)$, we accommodate densities that are just Lipschitz continuous, which naturally arise in intrinsic settings on the manifold, e.g., the intrinsic Gaussian distribution $p\propto \exp(-\frac{\rho(x,\Bar{x})^2}{2\sigma^2})$, $\Bar{x}\in M$.

\begin{assumption}
The kernel function $\kappa\in C^{(1,1)}(M)$.
\end{assumption}

This assumption ensures that all functions in its associated RKHS $\mathcal{H}_\kappa$ are continuously differentiable by Thm. \ref{kernel-differentiable}.

\vspace{5mm}

As previously stated, there is no global chart on a curved manifold, thus the partial derivatives $\{\frac{\partial}{\partial x^l}\}$ do not generalize to Riemannian manifolds globally. To have corresponding global derivatives, we should resort to the vector fields $D^l$ on manifolds instead, in which case the resulting Stein operator maps $\Vec{f}:=(f_1,\dots,f_m)\in\mathcal{H}^m_\kappa$ to $\sum_{l=1}^m \left[ D^l f_l+ f_l D^l\log p \right]$. However, one should note that the fact,   $\Div\frac{\partial}{\partial x^l}=0$ plays an important role in the Stein operator $\mathcal{A}_P$, since 
\begin{equation}\label{Derivation-E}
\begin{aligned}
\sum_{l=1}^d\Div(f_l p\frac{\partial}{\partial x^l} )
&=p \sum_{l=1}^d\left(\frac{\partial f_l}{\partial x^l}+f_l\frac{\partial}{\partial x^l}\log p+f_l \Div\frac{\partial}{\partial x^l}
\right) \\
&= p \mathcal{A}_p \Vec{f}+ p  \sum_{l=1}^d f_l\Div\frac{\partial}{\partial x^l}= p \mathcal{A}_P \Vec{f}, 
\end{aligned}
\end{equation}
The first equality is based on the property \cite[Exer. 2.5.5]{petersen2016riemannian} of the divergence operator $\Div(fD)=Df+f\Div D$. This equation leads to Stein's identity for $\mathcal{A}_p$ as 
\[ P(\mathcal{A}_p \Vec{f})=\int_{\mathbb{R}^d} p\mathcal{A}_p \Vec{f} dx= \sum_{l=1}^d\int_{\mathbb{R}^d} \Div(f_l p \frac{\partial}{\partial x^l})d x=0,\]
by divergence theorem (Thm. \ref{DivThm}). 

\emph{The above derivation process falls apart on manifolds, since $\Div D^l\neq 0$ in general and thus the last equality in Eq.\eqref{Derivation-E} fails to hold}. To preserve the Stein's identity on manifolds, we keep the term in Eq.\eqref{Derivation-E} that contains $\Div D^l$, and define the operator as:
\begin{tcolorbox}[colback=gray!5!white,colframe=gray!75!black,title=Stein's operator on $M$]

\begin{definition}[Stein operator on $M$] Given a group of vector fields $\{D^l\}_{l=1}^m$ on $M$, the \emph{Stein operator} $\mathcal{T}_p$ on $M$ is defined as
\begin{equation}\label{StOM}
\mathcal{T}_p:  \Vec{f}\mapsto \sum_{l=1}^m \left[ D^l f_l+ f_l D^l\log p + f_l \Div D^l \right], \quad \Vec{f}\in \mathcal{H}^m_\kappa. 
\end{equation}
Here, $D^l \log p$ is set to $0$ whenever $p$ is non-differentiable.
\end{definition}
\end{tcolorbox}

\begin{remark}  In contrast to $\mathbb{R}^d$, the number of vector fields $m$ here is not compelled to equal the dimension $d:=\dim M$ of $M$, as long as Stein's identity holds. In fact, we will usually need more vector fields than in the case of Euclidean space for manifolds, to preserve the properties that the Stein operator on Euclidean space possesses, as we will elaborate in \S\ref{KSD-Separability}.
\end{remark}

\begin{example}[Euclidean space] For the case $M=\mathbb{R}^d$, let $D^l=\frac{\partial}{\partial x^l}$, $1\leq l\leq d$. Since $\Div\frac{\partial}{\partial x^l}=0$, then Stein operator $\mathcal{T}_p$ on manifold \eqref{StOM} degenerate to the Stein operator $\mathcal{A}_p$ on Euclidean space \eqref{StOE}.
\end{example}

\begin{example}[Lie groups] The case where $M$ is a Lie group was presented in \cite{Qu2024kernel}. Let $D^l$, $1\leq l\leq d$ be the left-invariant vector fields on $G$, then $\Div D^l=D^l\Delta$, where $\Delta$ is the modular function. For illustrative examples, we refer the readers to our recent work \cite{Qu2024kernel}.   
\end{example}

\subsection{KSD on Riemannian Manifolds}\label{KSDonM}

With the Stein operator $\mathcal{T}_p$ in hand, we can now define the KSD on $M$ as follows:

\begin{definition}[KSD on Riemannian manifolds] Given the Stein operator $\mathcal{T}_p f$ in \eqref{StOM}, the \emph{kernel Stein discrepancy (KSD)} on $M$ is defined by plugging $\mathcal{T}_p$ into \eqref{KSDdef}, i.e., 
\begin{equation}\label{KSDdef-M}
    \ksd(P,Q):= \sup\big\{  Q(\mathcal{T}_p \Vec{f}): \Vec{f}\in\mathcal{H}^m_\kappa, \Vert \Vec{f}\Vert_{\mathcal{H}^m_\kappa}\leq 1 \big\}.
\end{equation}
\end{definition}

A remarkable property of KSD on $\mathbb{R}^d$ \cite{chwialkowski2016kernel,liu2016kernelized}, as well as the KSD on Lie groups \cite{Qu2024kernel} is that it has a closed form represented by an integral, which facilitates its use in the practical applications. 
The KSD on Riemannian manifold defined in \eqref{KSDdef-M} preserves this property, as we derive next. 

For notational convenience, we specify each component $\mathcal{T}^l_p$ of the Stein operator $\mathcal{T}_p$ as $\mathcal{T}^l_p: h\mapsto  D^l h+h D^l\log p+h \Div D^l$ for $h\in \mathcal{H}_\kappa$.
Thus, $\mathcal{T}_p \Vec{f}=\sum_{l=1}^m \mathcal{T}^l_p f_l$, for $\Vec{f}\in\mathcal{H}^m_\kappa$. We let $\mathcal{T}^l_p \kappa$ denote the bivariate function obtained by letting $\mathcal{T}^l_p$ act on the first argument of $\kappa$, let $(\mathcal{T}^l_p \kappa)_x$ represent the univariate function obtained by fixing the first argument of $\mathcal{T}^l_p \kappa$ at $x$, and let $(\Vec{\mathcal{T}}_p \kappa)_x$ represent the vector-valued function $((\mathcal{T}^1_p \kappa)_x,\dots,(\mathcal{T}^m_p \kappa)_x)$. 

By Thm. \ref{kernel-differentiable}, it is straightforward that $(\mathcal{T}^l_p\kappa)_x\in\mathcal{H}_\kappa$ and thus $(\Vec{\mathcal{T}}_p\kappa)_x\in\mathcal{H}^m_\kappa$ for all $x$. Therefore, $\phi^l_p:x\mapsto (\mathcal{T}^l_p\kappa)_x$ ($\Vec{\phi}_p:x\mapsto (\Vec{\mathcal{T}}_p\kappa)_x$ resp.) is a measurable map from $M$ to $\mathcal{H}_k$ ($\mathcal{H}^m_\kappa$ resp.). Furthermore, Thm. \ref{kernel-differentiable} implies that $\langle h,\phi^l_p(x)\rangle_{\mathcal{H}_\kappa}=\mathcal{T}^l_p h (x)$ for all $h\in\mathcal{H}_\kappa$, thus if we substitute the function $\phi^l_p(y)$ for $h$, we will have
\begin{equation}\label{kp}
\begin{aligned}
\kappa_p(x,y)&:=\langle \Vec{\phi}_p(x), \Vec{\phi}_p(y)) \rangle_{\mathcal{H}^m_\kappa}=\sum_{l=1}^m \langle \phi^l_p(y) , \phi^l_p(x) \rangle_{\mathcal{H}_\kappa}=\sum_{l=1}^m \mathcal{T}^l_{p(x)} \mathcal{T}^l_{p(y)} \kappa.  
\end{aligned}
\end{equation}
Here $\mathcal{T}^l_{p(x)} $ and $\mathcal{T}^l_{p(y)} $ represent the operators acting on the first argument $x$ of $\kappa$ and the second argument $y$ of $\kappa$ respectively, and $\mathcal{T}^l_{p(x)} \mathcal{T}^l_{p(y)} \kappa$ represent the bivariate function by letting $\mathcal{T}^l_{p(y)} $ act on $y$ first and then let $\mathcal{T}^l_{p(x)} $ act on $x$ of $\kappa$. Clearly, $\kappa_p(x,y)$ is symmetric and semi-positive definite.

Note that $\kappa_p(x,x)=\langle \Vec{\phi}_p (x), \Vec{\phi}_p(x) \rangle= \Vert \Vec{\phi}_p (x) \Vert^2$, thus if $\sqrt{\kappa_p(x,x)}$ is $Q$-integrable, then $\Vec{\phi}_p (x) $ is Bochner $Q$-integrable, whose Bochner integral is denoted by $Q(\Vec{\phi}_p)$. By Thm. \ref{Bochner},
 \begin{equation}
 \begin{aligned}
\ksd^2(P,Q)&= \sup_{\Vert \Vec{f}\Vert\leq 1} Q(\mathcal{T}_p \Vec{f})^2=\sup_{\Vert \Vec{f}\Vert\leq 1} Q[\langle \Vec{f},\Vec{\phi}_p\rangle]^2=\sup_{\Vert \Vec{f}\Vert\leq 1}\langle \Vec{f},Q(\Vec{\phi}_p)\rangle^2 =\Vert  Q(\Vec{\phi}_p)\Vert^2\\
&= \langle Q[\Vec{\phi}_p], Q[\Vec{\phi}_p]\rangle = (Q\times Q) [\langle \Vec{\phi}_p(\cdot), \Vec{\phi}_p(\cdot)\rangle] = \iint \kappa_p(x,y) Q(dx) Q(dy).
 \end{aligned}
\end{equation}
We summarize this result into following theorem:
\begin{theorem}[Closed form]\label{KSDFormThm}
Suppose $\sqrt{\kappa_p(x,x)}$ is $Q$-integrable, then the KSD on $M$ defined in Eq. (\ref{KSDdef-M}) satisfies
\begin{equation}\label{KSDForm}
\ksd^2(P,Q)=\iint \kappa_p(x,y) Q(dx) Q(dy). 
\end{equation}
Here $\kappa_p(x,y)$ is the function introduced in \eqref{kp}.
\end{theorem}

\begin{theorem}[Stein's identity]\label{Stein-identity} If $\sqrt{\kappa(x,x)}\sum_{l=1}^m|D^l|$,  $\sqrt{\kappa_p(x,x)}$ are $P$-integrable, then $P(\mathcal{T}_p\Vec{f})=0$ for all $\Vec{f}\in\mathcal{H}^m_\kappa$, i.e., $\ksd(P,P)=0$ or equivalently $Q=P\Rightarrow \ksd(P,Q)=0$.
\end{theorem}

\subsection{Separability of KSD}\label{KSD-Separability}

The Stein operator $\mathcal{T}_p$ in \eqref{StOM} satisfies Stein's identity, i.e., $P(\mathcal{T}_p \Vec{f})=0$ for $\Vec{f}\in \mathcal{H}^m_\kappa$, which is equivalent to stating that, $Q=P\Rightarrow \ksd(P,Q)=0$. However, in order to use KSD as a loss functions in practice, it should satisfy the reverse implication, i.e., $\ksd(P,Q)=0\Rightarrow Q=P$, in which case we say \emph{the KSD separates (discriminates)  $P$ (from $Q$)}. Moreover, if $\ksd(P,Q_n)\to 0$ implies $Q_n\Rightarrow P$ for a sequence of distribution $Q_n$, then we say $\ksd$ \emph{detects the weak convergence}. However, to ensure that the $\ksd$ separates  $P$, we must impose more conditions on the Stein operator and the Stein class.

\paragraph{Conditions on the Stein operator} If we recall the process of derivation of $\mathcal{A}_p$ on $\mathbb{R}^d$ in the original literature \cite{barp2019minimum,chwialkowski2016kernel,gorham2017measuring,liu2016kernelized,matsubara2022robust,oates2022minimum}, we discover that, intuitively, the component $\mathcal{A}^l_p f(x):=\frac{\partial f_l}{\partial x^l}(x)+f_l(x)\frac{\partial }{\partial x^l}\log p(x)$ of each $l$ in $\mathcal{A}_p \Vec{f}$, detects the difference between the slopes of $\log p$ and $\log q$ along the direction $
\frac{\partial}{\partial x^l} $ at $x$. Therefore, $\{\frac{\partial}{\partial x^l}\}_{l=1}^d$ must be a basis at each point of $\mathbb{R}^d$, so that $\mathcal{A}_p=\sum_{l=1}^d\mathcal{A}^l_p$ can detect the differences along all directions, so that we can conclude $\frac{\partial}{\partial x^l}\log p(x)=\frac{\partial}{\partial x^l}\log q(x) $ for all $1\leq l\leq d$ after some reasoning. In addition, to further conclude $P=Q$, we need the connectedness of $\mathbb{R}^d$. This motivates us to make the following additional assumptions:

\begin{assumption}\label{assumption-D}
For each $x\in M$, $D^1_x,\dots, D^m_x$ span the entire $T_x M$.
\end{assumption}

\begin{theorem}\label{vector-field-basis} Such a collection of vector fields that satisfies the standing assumption \ref{assumption-D} on $M$ always exists.
\end{theorem}

\paragraph{Conditions on the Stein class} Most widely-known results established by the past works (e.g., \cite{chwialkowski2016kernel,liu2016kernelized,xu2020stein}) is that, if $\kappa$ is \emph{$C_0$-universal}, then the $\ksd$ will separate $P$ from all $Q$ with $C^1$-densities. It was further strengthened in \cite[Thm. 3]{barp2024targeted} that, if $\kappa$ is a $C_0$-universal translation-invariant kernel on $\mathbb{R}^d$, then the $\ksd$ defined via the Stein operator \eqref{StOE} separates $P$ from all $Q$ such that $Q(|\nabla \log p|)<+\infty$, where $Q$ does not necessarily admit a density. Whereas our goal is to extend this result to the setting of manifold, there is no conception of "translation-invariant kernel" on a non-flat manifold. Nonetheless, we may embed the manifold $M$ into some Euclidean space $\mathbb{R}^{d'}$ ($d'>\dim M$) via a smooth embedding $\psi:M\to\mathbb{R}^{d'}$, and take the restriction $\Tilde{\kappa}(x,y):=\kappa(\psi(x),\psi(y))$ of a kernel $\kappa$ on $\mathbb{R}^{d'}$ to $M$. In practical applications, constructing kernels directly from the intrinsic geometry of a manifold has traditionally been a challenging task, thus it is common to obtain a kernel on the manifold by restricting one defined on the ambient space. This assumption does not substantially limit the practical range of admissible kernels.

\begin{theorem}\label{KSD-Characterization-Compact} Suppose $M$ is compact, $ p\in C^{s+1}(M)$ for some $s>\dim M/2$, and $\kappa$ is a characteristic translation-invariant kernel on $\mathbb{R}^{d'}$. Then the KSD defined via $\Tilde{\kappa}$ satisfies: $Q_n\Rightarrow P\iff \ksd(P,Q_n)\to 0$ for any sequence of distributions $Q_n$.
\end{theorem}

Thm. \ref{KSD-Characterization-Compact} significantly strengthens the existing results on manifolds in literature. The $\ksd$ in \cite{xu2020stein} only separated $P$ from $Q$ when they admit $C^1$ densities. \cite{barp2018riemann} requires $\kappa$ to generate a Sobolev space on $M$, which is considered limiting in practice since most of the commonly-used kernels (e.g., Gaussian, IMQ, Cauchy) do not fall into this category. By comparison, Thm. \ref{KSD-Characterization-Compact} requires only mild conditions so that all aforementioned commonly-used kernels are applicable. The compactness of the manifold not only prevents the blow-up of $D^l\log p$ at infinity but also ensures the tightness of $Q_n$, so that the KSD achieves the strongest separation result, i.e, detects the weak convergence, whereas the analogous result in the Euclidean setting \cite[Thm. 3 \& Thm. 9]{barp2024targeted} fails to do so without additional assumptions.

However, there exists a considerable number of non-compact manifolds and non-smooth densities that are commonly-used in practice, e.g., the on the manifold $\mathcal{P}(N)$ of symmetric positive definite matrices, the intrinsic Gaussian distribution $p\propto \exp(-\frac{\rho^2(x,\bar{x})}{2\sigma^2})$, etc. Therefore, we introduce a more general result that applies to locally Lipschitz continuous densities on non-compact manifolds next.

Let $\mathcal{P}_{p,\psi}:=\{Q\in\mathcal{P}(M): D^l\log p, |D^l|, |d\psi(D^l)|, 1\leq l\leq m \text{ are all } Q\text{-integrable} \}$. Here $d\psi$ is the pushforward (differential) of $\psi$ that maps the vector field $D^l$ to the vector field $d\psi(D^l)$ in $\mathbb{R}^{d'}$, and $|d\psi(D^l)|$ is the pointwise length of $d\psi(D^l)$ under the canonical Euclidean metric. Let $L^2(P)$ be the space of square-integrable functions w.r.t. $P$. If $Q$ is absolutely continuous w.r.t. $P$ (equivalent to $Q\ll \Omega$ as $P\sim\Omega$) and the R-N derivative $\frac{dQ}{dP}$ is in $L^2(P)$, and we simply write (with a slight abuse of notation) $Q\in L^2(P)$ and $\Vert Q\Vert_P:=\Vert \frac{dQ}{dP}\Vert_{L^2(P)}$.
\begin{theorem}\label{KSD-Characterization-Noncompact} Suppose $M$ is complete, $p$ is locally Lipschitz continuous, $\kappa$ is characteristic translation-invariant, $\Tilde{\kappa}_p(x,x)$ is $P$-integrable and $P\in\mathcal{P}_{p,\psi}$. Then we have 
\begin{itemize}
    \item[1.]  Given $Q\in\mathcal{P}_{p,\psi}\cap L^2(P)$, 
    $Q=P\iff \ksd(P,Q)=0 $.
    \item[2.] Given $Q_n\in\mathcal{P}_{p,\psi}\cap L^2(P)$ with $\sup_n\Vert Q_n\Vert_P<+\infty$, $Q_n\Rightarrow P\iff \ksd(P,Q_n)\to 0 $.
\end{itemize}
\end{theorem}

Although Thm. \ref{KSD-Characterization-Noncompact} relaxes the stringent assumption in prior works that the density of $Q$ must be differentiable, it still requires the absolute continuity of $Q$ w.r.t. $P$ (or $\Omega$). This is due to the assumption that $p$ is only locally Lipschitz continuous, thus $p$ is likely to be non-differentiable on some $P$-null (or $\Omega$-null) set $\mathcal{N}\subset M$. As $\ksd$ will ignore this set, $\ksd(P,Q)$ can still vanish if $Q(\mathcal{N})>0$. Therefore, it is reasonable to assume $Q\ll P$. 

If we impose additional conditions on the smoothness of $P$, it may be possible to establish stronger results to separate $P$ from distributions without densities. However, compared to the result in Euclidean setting (\cite[Thm. 3]{barp2024targeted}), the discussion on non-flat spaces will not only involve a highly complicated analysis of the decay rate of $p$ at infinity, but also depend on the specific shape of the embedding $\psi$.  Elaboration on such a result would take up a lot of (page) space and divert us from the main focus of this paper. Therefore, we will address this topic in our future work.

In addition to the stronger theorems $\ref{KSD-Characterization-Compact}$ and \ref{KSD-Characterization-Noncompact}, in the next theorem, we establish the analog of the 
classical result namely, KSD separates $P$ from $Q$ with differentiable density, for the Riemannian manifolds.

\begin{theorem}\label{KSD-Characterization-weak} Suppose $\kappa$ is $C_0$-universal (not necessarily translation-invariant) on $M$. Suppose further $P$ and $Q$ have locally Lipschitz densities $p$ and $q$ such that $\sqrt{\Tilde{\kappa}(x,x)}|D^l|$, $\sqrt{\Tilde{\kappa}_p(x,x)}$ and $D^l\log(p/q)$, $1\leq l\leq m$ are $Q$-integrable, then $Q=P\iff \ksd(P,Q)=0$.
\end{theorem}

Although Thm. \ref{KSD-Characterization-weak} only covers $Q$ with locally Lipschitz continuous density, it remains essential as it does apply to certain cases ruled out by Thm. \ref{KSD-Characterization-Noncompact}.

\begin{example} Consider the Gaussian distribution $p=\frac{dP}{d\Omega}\propto e^{-\frac{\Vert x\Vert^2}{2}}$ and the multivariate student t-distribution $q=\frac{dQ}{d\Omega}\propto (1+\Vert x\Vert^2)^{-\gamma}$, $\gamma>d/2+1$ on $\mathbb{R}^d$. Here $\frac{dQ}{dP}=\frac{q}{p}\notin L^2({P})$, thus Thm. \ref{KSD-Characterization-Noncompact} does not apply. However, $\nabla \log (p/q)= \frac{2\gamma x}{1+\Vert x\Vert^2}-x$ is $Q$-integrable, and other integrability conditions also hold, thus Thm. \ref{KSD-Characterization-weak} is applicable.
\end{example}

\subsection{Stein operator on Riemannian Homogeneous Spaces} \label{KSD-Homogeneous}

The definition of the Stein operator in \eqref{StOM} does not only require one to find a group of basis that satisfy standing assumption 4, but also requires one to compute the divergence of these vector fields, which is usually challenging for a general vector field, e.g., the one obtained in the proof of Thm. \ref{vector-field-basis}. However, in practice, most of the commonly encountered manifolds are \emph{Riemannian Homogeneous spaces}. In such spaces, we can select the vector fields $D^l$ as a special kind of vector fields, \emph{killing fields}, to get around such computational issues.

Before introducing the Riemannian Homogeneous space, first we introduce the notion of isometry and group action. An \emph{isometry} of a Riemannian manifold $M$ is a diffeomorphism from $M$ onto itself that preserves the distance. The isometry group $I(M)$ is the group of all isometries of $M$, which is a Lie group due to the Myers–Steenrod theorem \cite[Thm. 5.6.19]{petersen2016riemannian}. A \emph{group action} of a Lie group $G$ on a Riemannian manifold $M$ is an continuous group homomorphism $\Psi:G\to I(M)$ that maps each element $g$ to some isometry $\Psi_g$ on $M$. Conventionally, we omit the symbol $\Psi$, and denote the group action by $g.x:=\Psi_g(x)$. Note that $e.x=x$, where $e$ is the identity of $G$.
\begin{definition}
    A \emph{Riemannian homogeneous space} $H$, abbreviated as \emph{homogeneous space} in this work, is a Riemannian manifold such that there exists a Lie group $G$ that acts transitively on $H$, i.e., for each $x,y\in H$, there exists $g\in G$ such that $g.x=y$. 
\end{definition}
For a tangent vector $E\in T_e G$, we take a local curve $\mathfrak{e}(t)$ in the equivalence class of $E$. Since $e.x=x$, thus, $\mathfrak{e}(t).x$ is a local curve on $H$ at $x$, and thus corresponds to a tangent vector in $T_x H$. For each $x\in H$, $\mathfrak{e}(t).x$ corresponds to a tangent vector, thus they form a vector field, denoted by $K$. Such a vector field $K$ is said to be a \emph{killing field} of $G$ on $H$. This correspondence is linear, i.e., if $E^1$ and $E^2$ corresponds to $K^1$ and $K^2$ respectively, then $a E^1+b E^2$ corresponds to $a K^1+ b K^2$. Specifically, killing fields are \emph{divergence-free}, i.e., $\Div K=0$.

\begin{theorem}\label{sub-killing-field} Suppose Lie group $G$ acts transitively on homogeneous space $H$. For a basis $E^1,\dots,E^m$ of $T_e G$ ($m=\dim G$), they correspond to a group of killing field $K^1,\dots,K^m$ in the way previously introduced. Then $K^1,\dots,K^m$ is a group of divergence-free vector fields on $H$ that satisfies standing assumption \ref{assumption-D}.
\end{theorem}

In this case, the \emph{Stein operator} $\mathcal{T}_p$ becomes
\begin{tcolorbox}[colback=gray!5!white,colframe=gray!75!black,title=Stein Operator on Homogeneous Spaces]
\begin{equation}\label{StoH}
\mathcal{T}_p : \Vec{f}\mapsto \sum_{l=1}^m \left[ K^l f_l+ f_l K^l\log p \right], \quad \Vec{f}\in \mathcal{H}^m_\kappa. 
\end{equation}
\end{tcolorbox}

\subsection{Empirical kernel Stein discrepancy}\label{Emp-KSD}

The integral closed form (\ref{KSDForm}) is one of the most significant properties of KSD. However, it may not be computable in practice due to following commonly-encountered situations:

\subsubsection*{Only samples available} In practice, 
$Q$ may be accessible only via samples instead of its exact form of distribution being known. This is commonly encountered in parameter estimation problems where, it is common to use a parameterized density family $p_\theta$ to approximate an unknown distribution $Q$, merely from its samples.

\subsubsection*{Intractable Integral} Sometimes, we do have the exact form of $Q$ but the integral in (\ref{KSDForm}) is intractable. For example, in the rotation tracking problem encountered in robotics \cite{suvorova2021tracking}, one must approximate the posterior distribution of \(Q_k\) with a von Mises-Fisher distribution so that the tracking algorithm (Kalman filter) updates consistently lie in the same space. In a Bayesian fusion problem \cite{lee2018bayesian}, a similar situation arises when it is required to guarantee that the result of the fusion stays in the family. 

\vspace{5mm}

These above two situations also arise in the KL-divergence setting,  $\KL(p,q)=\mathbb{E}_q[\log p]-\mathbb{E}_q[\log q]$, and in practice, it is common to use the empirical KL-divergence. For example, suppose we aim to approximate an unknown density $q$ with a density family $p_\alpha$ using the KL-divergence, but only have samples $x_i$ from $q$ instead of its density. We then minimize the empirical KL-divergence $n^{-1}\sum_{i} \log p_\alpha(x_i)$ as an alternative ($\mathbb{E}_q[\log q]$ is constant w.r.t. $\alpha$), which converges to $\mathbb{E}_q[\log p]$ almost surely by the law of large number. The minimizer of the empirical KL-divergence $n^{-1}\sum_{i} \log p_\alpha(x_i)$ is exactly the maximum likelihood estimator.

Analogously, the KSD has empirical versions, i.e., the $U$- and $V$-statistics:
\begin{equation}\label{U&V}
U_n:= \frac{1}{n(n-1)} \sum_{i\neq j} \kappa_p(x_i,x_j),\quad  V_n:= \frac{1}{n^2}\sum_{i,j} \kappa_p(x_i,x_j),
\end{equation}
which can serve as the alternatives to the integral KSD in \eqref{KSDForm} when we only have samples from $Q$. Several prior research works \cite{chwialkowski2016kernel,liu2016kernelized,xu2021interpretable} have developed a kernel Stein goodness of fit test based on the $U$-statistics.

If we do have the exact form of $Q$ but the integral in \eqref{KSDForm} is intractable, then we can draw samples from $Q$ with various sampling algorithms, e.g., Hamiltonion Monte Carlo or Metropolis-Hastings algorithms and adopt $U_n$ or $V_n$ as the measurement of the dissimilarity between $P$ and $Q$. Alternatively, if these sampling methods are hard to implement, we could use the importance sampling scheme to sample from another easy-to-sample distribution $\Lambda$ such that $Q\ll\Lambda$ and consider the following \emph{weighted empirical KSD}:
\begin{equation}\label{Uw&Vw}
U^w_n:= \frac{1}{n(n-1)} \sum_{i\neq j} \kappa^w_p(x_i,x_j) ,\quad V^w_n:= \frac{1}{n^2}\sum_{i,j} \kappa^w_p(x_i,x_j),
\end{equation}
where $\kappa^w_p$ is the weighted $\kappa_p$ function given by
\begin{equation}
\kappa^w_p(x,y):=\kappa_p(x,y) q^w(x)q^w(y), \quad \text{ where }\ q^w=\frac{dQ}{d\Lambda}\  \text{(R-N derivative)}.
\end{equation}
If $\Lambda=Q$, then $\kappa^w_p$ will degenerate to $\kappa_p$, thus $U^w_n$ and $V^w_n$ will degenerate to $U_n$ and $V_n$. 

\begin{example} Let $Q$ be the Riemannian Gaussian distribution on $\mathbb{S}^{N-1}$ with density $q\propto \exp\left(-\frac{\rho^2(x,\Bar{x}_0)}{2}\right)$ for some fixed $\Bar{x}_0\in\mathbb{S}^{N-1}$. Let the easy-to-sample distribution $\Lambda$ be the uniform distribution on $\mathbb{S}^{N-1}$, then the weighted empirical KSDs between $P$ and $Q$ are:
\[
U^w_n=\frac{1}{n(n-1)} \sum_{i\neq j} \kappa_p(x_i,x_j) q(x_i) q(x_j),\quad V^w_n=\frac{1}{n^2} \sum_{i,j} \kappa_p(x_i,x_j) q(x_i) q(x_j).
\]
Note that the empirical KSDs are proportional to the normalizing constant of $q$.
\end{example}

Similar to the empirical KL-divergence, the empirical KSD will converges to KSD almost surely, as stated in the next theorem:


\begin{theorem}\label{KSDasymptotic} Suppose $M$ is a Riemannian manifold and $\kappa^w_p(x,x)$ is $\Lambda$-integrable. Given $M$-valued samples $x_i\sim \Lambda$, we have
\begin{enumerate}
    \item $U^w_n,V^w_n\xrightarrow{\text{a.s.}}\ksd^2(P,Q)$ with the rate of $O_p(n^{-1})$,
    \item If $P\neq Q$, then $\sqrt{n}[U^w_n-\ksd^2(P,Q)]$ and $\sqrt{n}[V^w_n-\ksd^2(P,Q)]$ both converge to $N\big(0,\Tilde{\sigma}^2\big)$ in distribution, where $\Tilde{\sigma}^2:=4\text{var}_{x'\sim \Lambda}[\mathbb{E}_{x\sim \Lambda} \kappa^w_p(x,x')]$.
    \item If $P = Q$, then $n U^w_n$ converges to $\sum_{k=1}^\infty \lambda_k (Z^2_k-1)$, $n V^w_n$ converges to $\sum_{k=1}^\infty \lambda_k Z^2_k$ in distribution, where $\lambda_k$ are the eigenvalues of $\kappa^w_p(x,y)$ as introduced in Thm. \ref{Mercer} and $Z_k$ are i.i.d. standard Gaussian random variables.
\end{enumerate}
\end{theorem}

\section{Minimum kernel Stein discrepancy estimator } \label{MKSDE}

Since KSD measures the dissimilarity between distributions, one may natually use it on distributional approximation. Suppose we want to approximate a distribution $Q$ with a parametrized family $P_\alpha$, then we define the global minimizer
\begin{equation}\label{MKSDE-def}
\widehat{\alpha}:=\argmin\ksd(P_\alpha,Q)
\end{equation}
as the \emph{minimum kernel Stein discrepancy estimator} (MKSDE). However, as we explained in \S\ref{Emp-KSD}, $\ksd(P_\alpha,Q)$ may not be computable in most of the situations in practice. Therefore, we minimize empirical KSDs over $\alpha$ based on various situations:
\[
U^w_n(\alpha):= \frac{1}{n(n-1)} \sum_{i\neq j} \kappa^w_\alpha(x_i,x_j) ,\quad  V^w_n(\alpha):= \frac{1}{n^2}\sum_{i,j} \kappa^w_\alpha(x_i,x_j),
\]
where $\kappa^w_\alpha(x,x'):=\kappa^w_{p_\alpha}(x,x')$. Here $U^w_n$ and $V^w_n$ accommodate the unweighted case $\Lambda=Q$.

Although $U$-statistics is far more frequently mentioned in prior works \cite{barp2019minimum,chwialkowski2016kernel,liu2016kernelized, xu2021interpretable} as it is an unbiased estimator of $\ksd^2(P_\alpha,Q)$, we notice that the $V$-statistics exhibits better stability for optimization, as explained in the next section studying the asymptotic properties of MKSDE.

\subsection{Asymptotic Properties of MKSDE}\label{MKSDE-asymp}

To obtain the asymptotic behavior of the MKSDE, we re-tag the index $\alpha$ of the density family by $\theta$, and assume that $\theta$ is from some topological space $\Theta$. We denote $\ksd(\theta):=\ksd(P_\theta,Q)$ and let $\kappa^w_\theta(x,x'):=\kappa^w_{p_\theta}(x,x')$. Let \(\Theta_0\subset \Theta\) be the set of best approximators \(\theta_0\), i.e., \(\ksd(\theta_0)=\inf_\theta \ksd(\theta)\). With the additional topological structure on $\Theta$, we can establish stronger asymptotic results of $U^w_n$ and $V^w_n$. The asymptotic results in this section are satisfied by both $U^w_n$ and $V^w_n$. For notational convenience, we let $W_n(\theta)$ denote whichever $U^w_n(\theta)$ or $V^w_n(\theta)$.


\begin{theorem} \label{KSDasymptotic-stronger} If \(\kappa^w_{(\cdot)}(\cdot,\cdot)\) is jointly continuous and \(\sup_{\theta\in K} \kappa^w_\theta(x,x)\) is $\Lambda$-integrable for any compact \(K\subset\Theta\), then $W_n(\theta)\to\ksd^2(\theta)$ compactly almost surely, i.e., the following event 
\[
W_n(\theta) \to \ksd^2(\theta) \text{ uniformly on any compact }  K\subset\Theta
\]
almost surely happens. As a corollary, $\ksd(\cdot)$ is continuous if \(\Theta\) is locally compact.
\end{theorem}

Let \(\widehat{\Theta}_n\) be the set of MKSDE, i.e.,  global minimizers $\hat{\theta}_n$ of $W_n$, which is a random set. We can establish the strong consistency of MKSDE with Thm. \ref{KSDasymptotic-stronger} in hand. 


\begin{theorem}\label{MKSDEconsistency} Suppose all the conditions in Thm. \ref{KSDasymptotic-stronger} hold and suppose $\Theta$ satisfies one of following three conditions:
\begin{enumerate}
    \item $\Theta$ is compact;
    \item $\Theta$ is a geodesic convex subset of a Riemannian manifold, $W_n$ is convex on $\Theta$, $\Theta_0$ is non-empty, compact and $\Theta_0\subset\mathring{\Theta}_2$(interior);
    \item $\Theta=\Theta_1\times\Theta_2$, where $\Theta_1$ is compact, and for each fixed
    $\theta_1\in\Theta_1$, $\{\theta_1\}\times\Theta_2$ and $W_n(\theta_1,\cdot)$ satisfy the second condition.
\end{enumerate}
Then \(\Theta_0\), \(\widehat{\Theta}_n\) are non-empty for large \(n\) and \(\sup_{\theta\in \widehat{\Theta}_n} \rho(\theta,\Theta_0)\to 0\) almost surely.
\end{theorem}

It is worth noting that if $Q$ is a member of the family $P_\theta$, then the global minimizer set $\Theta_0$ is a singleton $\{\theta_0\}$. In such a situation, the MKSDE always converges to the unique ground truth $\theta_0$. Moreover, if the parameter space $ \Theta:=  \Theta_1\times\cdots\times\Theta_m $ is multi-dimensional, we combine all compact components into one compact parameter space, and hence the convex components as well, so that Thm. \ref{KSDasymptotic-stronger} is still applicable.

Existing result in literature \cite[Thm. 3.3]{barp2019minimum} showed the consistency of MKSDE, assuming that the parameter space $\Theta$ is either compact or satisfies the conditions of convexity, which is not applicable to the case where there are compact and convex parameters simultaneously. For example, consider the Riemannian Gaussian distribution $p\propto \exp(-\frac{\rho^2(x,\bar{x})}{2\sigma^2})$ on a compact manifold, if we redenote $\varsigma:=\sigma^{-2}$, then $\mu$ is a "compact" parameter and $\varsigma$ is a "convex" parameter.
 
To establish the asymptotic normality of MKSDE, we assume that \(\Theta\) is a connected Riemannian manifold with the Riemannian logarithm map $\Log$. We assume that \(\Theta_0\) and \(\widehat{\Theta}_n\) are non-empty for \(n\) large enough, and $\hat{\theta}_n$ is a sequence of MKSDE that converges to one of the global minimizer $\theta_0$ of $\ksd(\theta)$. Additionally, we assume the following conditions:
\begin{enumerate}
    \item[\textbf{(A1)}] $\kappa^w_\theta(x,y)$ is jointly continuous, and twice continuously differentiable in $\theta$.
    \item[\textbf{(A2)}] there exists a compact neighborhood $K$ of $\theta_0$ s.t. $\sup_{\theta\in K}\Vert \nabla \kappa^w_{\theta_0}\Vert$ is $\Lambda\times \Lambda$-integrable.
    \item[\textbf{(A3)}]$\Vert \nabla \kappa^w_{\theta_0}\Vert^2$ and $\Vert\mathcal{I}_{\theta_0}\Vert$ are $\Lambda\times \Lambda$-integrable, $\Vert\mathcal{I}_{\theta_0}(x,x)\Vert$ is $\Lambda$-integrable.
    \item[\textbf{(A4)}] $\mathcal{I}_{\theta_0}(x,y)$ is equi-continuous at $\theta_0$. 
    \item[\textbf{(A5)}] $\Gamma:=\frac{1}{2}\mathbb{E}_{x,y\sim w} [\mathcal{I}_{\theta_0}(x,y)]$ is invertible.
\end{enumerate}
Here $\nabla \kappa^w_\theta(x,y)$ represents the gradient of $\kappa^w_\theta(x,y)$ w.r.t $\theta$, and $\mathcal{I}_\theta(x,y)$ represents the Hessian of $\kappa^w_\theta(x,y)$ w.r.t $\theta$. In addition, let $\Sigma$ be the covariance matrix of the random vector $\mathbb{E}_{Y\sim w}[\kappa^w_{\theta_0}(x,Y)]$.


\begin{theorem}\label{MKSDECLT} Under assumptions $\textbf{A1}\sim\textbf{A5}$, 
 \(\sqrt{n}\Log_{\theta_0} (\hat{\theta}_n) \xrightarrow{d.} \mathcal{N}(0,\Gamma^{-1}\Sigma\Gamma^{-1})\). 
\end{theorem}

\subsection{MKSDE Composite goodness of Fit Test}\label{MKSDE-GoF}

KSD is commonly used to develop normalization-free goodness of fit tests \cite{chwialkowski2016kernel,liu2016kernelized}, which test whether a group of samples can be well-modeled by a given distribution $P$. More precisely, if we denote by $Q$ the unknown underlying distribution of samples, then we aim to test $H_0:P=Q$ versus $H_1:P\neq Q$. However, in many applications the candidate distribution for the given samples is usually not a specific distribution but a parameterized family $P_\theta$, while most existing methods only apply to a specific candidate, and requires testing individually for each member of $P_\theta$. Recently, a \emph{composite goodness-of-fit test} was developed to test whether a group of samples matches a family $P_\theta$ \cite{key2021composite}, i.e., test $H_0: \exists\theta_0, P_{\theta_0}=Q $ versus $H_1: \forall\theta, P_\theta\neq Q$, and was later generalized to all Lie groups in \cite{Qu2024kernel}.
This however was valid for samples of distributions over Lie groups and {\it not general Riemannian manifolds}. In this section, {\it we will generalize the composite goodness of fit test to all Riemannian manifolds}, develop a one-shot (direct) method, using the MKSDE obtained by minimizing (\ref{Uw&Vw}).

To implement the test, we assume the conditions in Thm. \ref{MKSDEconsistency} hold, so that \(\Theta_0\) and \(\widehat{\Theta}_n\) are non-empty. We also assume that \(\Theta_0\) is a singleton so that the MKSDE \(\hat{\theta}_n=\argmin \wksd^2_n(\theta)\) converges to the unique ground truth \(\theta_0\). We follows the notation in \ref{MKSDE-asymp} letting $W_n$ denote whichever $U^w_n$ or $V^w_n$.

Under the null hypothesis $H_0$, $n W_n(\theta_0)$ converges to $\sum_k \lambda_k (Z^2_k-1)$ or $\sum_k \lambda_k Z^2_k$ in distribution asymptotically by Thm. \ref{KSDasymptotic}. Let $\gamma_{1-\beta}$ be the $(1-\beta)$-quantile of $\sum_k \lambda_k (Z^2_k-1)$ or $\sum_{k=1}^\infty \lambda_k Z^2_k$ with significance level $\beta$.  We reject $H_0$ if $n W_n(\hat{\theta}_n)\geq \gamma_{1-\beta}$, as it implies $n W_n(\theta_0)\geq n W_n(\hat{\theta}_n) \geq \gamma_{1-\beta} $, since $\hat{\theta}_n$ is the minimizer of $W_n$. 

As there is no method in general to directly compute the infinite eigenvalues $\lambda_k$-s of $\kappa^w_{\theta_0}$, Gretton et al. \cite{gretton2009fast} introduced a method to approximate $\lambda_k$-s by the eigenvalues of the empirical matrix $G^w_n(\theta_0):=n^{-1} (\kappa^w_{\theta_0}(x_i,x_j))_{i j}$. Note that the ground truth $\theta_0$ is unknown in our setting, but the minimizers $\hat{\theta}_n$ converge to $\theta_0$ almost surely under specific conditions by Thm. \ref{MKSDEconsistency}. Therefore, we may approximate $\lambda_k$ by the eigenvalues of the empirical matrix $G^w_n(\hat{\theta}_n):=n^{-1}(\kappa^w_{\hat{\theta}_n}(x_i,x_j))_{i j}$. Let $\hat{\lambda}_k$ be the eigenvalues of the matrix $G^w_n(\hat{\theta}_n)$ (set $\hat{\lambda}_k:=0$ for $k>n$), then we have

\begin{theorem}\label{GoF}  Under the conditions in Thm. \ref{MKSDEconsistency}, $\sum_{k=1}^\infty (\hat{\lambda}_k-\lambda_k) Z^2_k\to 0$ in probability.
\end{theorem}

Therefore, the $\sum_{k=1}^n \hat{\lambda}_k (Z^2_k-1)$ and $\sum_{k=1}^n \hat{\lambda}_k Z^2_k$ can serve as an empirical estimate of the asymptotic distribution $\sum_{k=1}^\infty \lambda_k (Z^2_k-1)$ and $\sum_{k=1}^\infty \lambda_k Z^2_k$. The MKSDE goodness of fit test algorithm is summarized in the following algorithm block \ref{GoF-Algo}.

\begin{algorithm}
\caption{MKSDE goodness of fit test}\label{GoF-Algo}
\textbf{Input}: population $x_1,\dots,x_n\sim w$; sample size $n$; number of generations $n'$; significance level $\beta$.

 \textbf{Test:} $H_0:\exists\theta_0, P_{\theta_0}=Q$ versus $H_1:\forall\theta,\    P_\theta \neq Q$.

\textbf{Procedure:}

\Indp 1. Obtain minimizer $\hat{\theta}_n$ of $W_n(\theta)$ in \eqref{Uw&Vw}.

 2. Obtain the eigenvalues $\hat{\lambda}_1,\dots, \hat{\lambda}_n$ of $G^w_n(\hat{\theta}_n):=n^{-1}(\kappa^w_{\hat{\theta}_n}(x_i,x_j))_{i j}$.

 3. Sample $Z^l_k\sim N(0,1)$, $1\leq k\leq n$, $1\leq l\leq n'$ independently.

 4. Compute $\gamma^l=\sum_{k=1}^n \hat{\lambda}_k [(Z^l_k)^2-1] $ or $\sum_{k=1}^n \hat{\lambda}_k (Z^l_k)^2 $.

 5. Determine estimation $\hat{\gamma}_{1-\beta}$ of $(1-\beta)$-quantile using $\gamma^1,\dots,\gamma^{n'}$.

 \Indm

 \textbf{Output:} Reject \(H_0\) if \(n W_n(\hat{\theta}_n)>\hat{\gamma}_{1-\beta}\). 
\end{algorithm}
\section{Closed Form KSD on Homogeneous Spaces}\label{Examples}

As we can see from our earlier discussion, the calculation of 
$\kappa_p$ function from \eqref{kp} plays an important role in the practical usage of KSD and MKSDE. Suppose $H$ is a homogeneous space, and $\{K^l\}_{l=1}^m$ is a killing field basis, $\kappa$ is kernel function and $p$ is a density on $H$. After some straightforward calculations, we obtain 
\begin{equation}\label{kp-general}
\begin{aligned}
\kappa_p(x,y)=\sum^m_{l=1}\mathcal{T}^l_{p(y)}\mathcal{T}^l_{p(x)}\kappa= \kappa\cdot \sum^m_{l=1} \left[K^l_x\log(p \kappa)\cdot K^l_y\log(p \kappa) + K^l_y K^l_x\log \kappa\right].
\end{aligned}    
\end{equation}
In this section, we will present a surprising result namely that, the formula \eqref{kp-general} of $\kappa_p$ function can be further simplified into closed form expressions in most commonly-encountered homogeneous spaces, including:
\begin{itemize}
    \item the Stiefel manifolds $\mathcal{V}_r(N):=\{X\in\mathbb{R}^{N\times r}: X^\top X=I_{r\times r} \}$, inheriting the subspace topology from $\mathbb{R}^{N\times r}$, including the sphere $\mathcal{V}_1(N):=\mathbb{S}^{N-1}$, the special orthogonal group $\mathcal{V}_{N-1}(N)=\SO(N)$ and the orthogonal group $\mathcal{V}_{N}(N)=\Or(N)$. 
    \item the Grassman manifold $\mathcal{G}_r(N):=\{X\in\mathbb{R}^{N\times N}: X^\top=X, X^2=X, \tr(X)=r \}$, the space of all $r$-planes in $\mathbb{R}^N$, or equivalently, the space of orthogonal projections of $\mathbb{R}^N$ into $r$-dimensional subspaces, or equivalently, the space of $N$ by $N$ idempotent symmetric matrices with rank $r$. 
    \item $\mathcal{P}(N):\{X\in\mathbb{R}^{N\times N}: X^\top=X, X\succ 0\}$, the manifold of $(N,N)$ symmetric positive definite (SPD) matrices, inheriting the subspace topology from $\mathbb{R}^{N\times N}$.
\end{itemize}
Even more surprisingly, the MKSDE obtained by minimizing the empirical KSDs \eqref{U&V} and \eqref{Uw&Vw}, also has closed form expressions, if $p_\theta$ is the exponential family given by $
    p(x|\theta)\propto \exp(\theta^\top \zeta(x)+\eta(x) ),\quad \theta\in \mathbb{R}^{s},\  $
where $\zeta:=(\zeta_1,\dots,\zeta_s)^\top\in C^1(H,\mathbb{R}^s)$, $\eta\in C^1(H,\mathbb{R})$ for some $s\in\mathbb{N}_+$. We plug this into \eqref{kp-general} and get
\begin{equation}
\begin{aligned}
\kappa_p(x,y)&= \theta^\top\cdot \kappa\sum_{l=1}^m K^l_x\zeta\cdot K^l_{y}\zeta^\top\cdot\theta + \kappa\sum_{l=1}^m K^l_x\log (e^\eta\kappa) K^l_{y}\zeta^\top \theta\\
&+ \kappa\sum_{l=1}^m K^l_y\log (e^\eta\kappa) K^l_x\zeta^\top \theta+ c(x,y),
\end{aligned}  
\end{equation}
where $c(x,y)$ denote the terms independent of $p_\theta$ thus independent of its parameter $\theta$. Note that $K^l_x\zeta$ is a $s$-dimensional vector, thus $K^l_x\zeta\cdot K^l_y\zeta^\top$ is the matrix $(K^l_x\zeta_i\cdot K^l_y\zeta_j)_{i j}$. Furthermore, we let
\begin{equation}\label{Q&b}
Q(x,y)= \kappa\sum_{l=1}^m K^l_x\zeta\cdot K^l_{y}\zeta^\top,\quad b(x,y)= \kappa\sum_{l=1}^m K^l_x\log (e^\eta\kappa) \cdot K^l_y\zeta.
\end{equation}
Therefore, the empirical KSDs in \eqref{U&V} and (\ref{Uw&Vw}) will be quadratic forms of $\theta$:
\begin{equation}\label{Uw&Vw-exp}
U^w_n(\theta) = \theta^\top Q_u \theta + 2 b_u \theta +c,\quad V^w_n(\theta) = \theta^\top Q_v \theta + 2 b_v \theta +c,
\end{equation}
where
\begin{equation}\label{Qu&bu&Qv&bv}
\begin{aligned}
Q_u &:= \sum_{i\neq j}\frac{Q(x_i,x_j)}{n(n-1)} q^w(x_i) q^w(x_j),\quad b_u:= \sum_{i\neq j}\frac{ b(x_i,x_j) }{n(n-1)}q^w(x_i) q^w(x_j),\\
Q_v &:= \sum_{i,j}\frac{Q(x_i,x_j) }{n^2}q^w(x_i) q^w(x_j),\quad b_v:=\sum_{i,j}\frac{ b(x_i,x_j) }{n^2 }q^w(x_i) q^w(x_j).
\end{aligned}
\end{equation}
It is noteworthy that $Q_u$, $Q_v$ are both symmetric. Furthermore, $Q_v$ is always semi-positive definite, thus $V_n(\theta)$ can always attain its minimum values for exponential family, while $U_n$ can not. Additionally, $Q_u$ and $Q_v$ are not necessarily invertible, e.g. $p_\theta$ is not identifiable. In such cases, the global minimum point can be represented by the Moore–Penrose inverse $Q_u^+$ and $Q_v^+$ of $Q_u$ and $Q_v$. We summarize into the following theorem:

\begin{theorem}\label{MKSDE-Form-Thm} $Q_u$, $Q_v$ defined in \eqref{Qu&bu&Qv&bv} are both symmetric, and $Q_v$ is positive semi-definite. If further $Q_u$ is positive semi-definite, then the global minimizer sets of $U^w_n(\theta)$ and $V^w_n(\theta)$ in \eqref{Uw&Vw-exp} can be represented by the Moore-Penrose inverse $Q_u^+$, $Q_v^+$ of $Q_u$, $Q_v$ as follows
\begin{equation}\label{MKSDE-Form}
\begin{aligned}
\argmin U^w_n(\theta) = \{-Q_u^+ b_u - (I - Q_u^+ Q_u)x: x\in\mathbb{R}^s  \},\\
\argmin V^w_n(\theta) = \{-Q_v^+ b_v - (I - Q_v^+ Q_v)x: x\in\mathbb{R}^s  \}.
\end{aligned}
\end{equation}
For the unweighted case, we just ignore the weighted ratio $q^w(x_i) q^w(x_j)$ in \eqref{Qu&bu&Qv&bv}.
\end{theorem}

To obtain the different forms of MKSDE, it suffices to derive the closed form of $Q(x,y)$ and $b(x,y)$ in \eqref{Q&b} on different manifolds. In this section, we will explicitly calculate aforementioned closed forms of KSD and MKSDE on commonly-encountered homogeneous spaces, and provide multiple examples for specific families and specific choice of kernels. 

\subsection{Matrix Algebra}
\label{Mat-A}

In this section, since the sample space $M$ is taken to be a matrix manifold, we first introduce some background matrix algebra along with convenient notation that will be used subsequently in this section. We refer the reader to \cite{Graham} and \cite{magnus2019matrix} for more details on this topic.

\subsubsection*{Matrix Inner Product space}
The Frobenius inner product
$\langle A,B\rangle_{\F}:=\tr(A^\top B)=\sum_{i=1}^N\sum_{j=1}^r a_{ij} b_{ij},\quad \text{for } A=(a_{ij})\in \mathbb{R}^{N\times r} \text{ and } B=(b_{ij})\in \mathbb{R}^{N\times r},
$
is usually considered to be the canonical inner product on $\mathbb{R}^{N\times r}$. The Frobenius norm is given by $\Vert A\Vert_{\F}:=\sqrt{\langle A,A\rangle_{\F}}$. We let $\vectz(A)$ denote the vectorization of $A$ obtained by stacking the columns, and define then following functions 
$ \mathscr{S}(X)=(A+A^\top)/2,\quad \mathscr{A}(A)=(A-A^\top)/2,$
be the symmetrization and skew-symmetrization of a squared matrix $A\in\mathbb{R}^{N\times N}$. Let $\mathscr{S}(\mathbb{R}^{N\times N})$ and $\mathscr{A}(\mathbb{R}^{N\times N})$ be the linear subspaces of all symmetric and skew-symmetric $N\times N$ matrices respectively. It can be easily checked that they are the orthogonal complements of each other, and $X=\mathscr{S}(X)+\mathscr{A}(X)$
is the orthogonal decomposition of $X$ onto $\mathscr{S}(\mathbb{R}^{N\times N})$ and $\mathscr{A}(\mathbb{R}^{N\times N})$.

Let $E_{i j}$ be the matrix prescribed with all zeros everywhere except a $1$ at the $(i,j)^\text{th}$ entry. Specifically, when $r=1$, we set $e_i$ be the $N\times 1$ vector prescribed with all zeros except $1$ at the $i^\text{th}$ element. Let $\mathcal{E}_{i j}:=\left( \frac{\sqrt{2}}{2} E_{i j}-\frac{\sqrt{2}}{2}E_{j i}\right) $. Then following proposition is easy to verify: 

\begin{proposition}\label{summation}
$\{E_{i j}, 1\leq i,j\leq N\}$ and $\{\mathcal{E}_{i j},1\leq i<j\leq N\}$ are orthonormal bases of $\mathbb{R}^{N\times N}$ and $\mathscr{A}(\mathbb{R}^{N\times N})$ respectively, and for any $A,B\in\mathbb{R}^{N\times N}$ we have
\[
\sum_{i,j}\langle E_{i j}, A\rangle_{\F}\cdot \langle E_{i j},B\rangle_{\F}=\langle A,B\rangle_{\F},\quad \sum_{i<j}\langle \mathcal{E}_{i j}, A\rangle_{\F}\cdot \langle \mathcal{E}_{i j},B\rangle_{\F}=\langle\mathscr{A}(A)^\top, \mathscr{A}(B)\rangle_{\F}.
\]
\end{proposition}

\subsubsection{Gradient of functions on a matrix manifold}\label{Grad-Mat}

Suppose $M$ is a sub-manifold of $\mathbb{R}^{N\times r}$, i.e., a matrix manifold whose elements are $N\times r$ matrices, and $M$ is endowed with the Riemannian metric $g_{X}(\cdot,\cdot)$ for $X\in M$.  

For a real-valued function $f$ on $M$, the \emph{Riemannian gradient} of $f$ at $X$ is defined as the tangent vector $\nabla^M_X f\in T_X M $, such that $D_X f=g_X(D_X,\nabla^M_X f) $ for all tangent vectors $D_X\in T_X M$. Analogously, since tangent space $T_X M$ at each point $X\in M$ is a linear subspace of $\mathbb{R}^{N\times r}$, we may define the \emph{Euclidean gradient} of $f$ w.r.t the Frobenius inner product $\langle\cdot,\cdot\rangle_{\F}$, i.e., a matrix $\nabla^\mathbb{R}_X f\in\mathbb{R}^{N\times r}$ such that $D_X f=\langle \nabla^\mathbb{R}_X f, D_X \rangle_{\F}$ for all $ D_X\in T_X M\subset \mathbb{R}^{N\times r}$. 

Note that $\nabla^\mathbb{R}_X f$ is not necessarily an element in $T_X M$, thus the Euclidean gradient is not unique in general. For example, consider function $f(x) = \mu^\top x$ for $x$ on the sphere $\mathbb{S}^{N-1}$. Since the Riemannian gradient must lie in $T_x\mathbb{S}^{N-1}$, we have  $\nabla^M_x f=\mu-(\mu^\top x)x$, the orthogonal projection of $\mu$ onto the $T_x\mathbb{S}^{N-1}$. On the other hand, as $x\perp T_x\mathbb{S}^{N-1}$, the Euclidean gradient $\nabla^\mathbb{R}_x f$ can be any vector with the form $\mu-\alpha x$, $\alpha\in\mathbb{R}$.

In this work, we address this notion only for computational and representational convenience, but it is not widely used in other works due to non-uniqueness. As different choices of $\nabla^\mathbb{R}_X f$ will deliver the same value of $\langle \nabla^\mathbb{R}_X f, D_X \rangle_{\F}$ for any $D_X\in T_X M$, and they only appear inside the inner product bracket $\langle\cdot,\cdot\rangle_{\F}$ in this work, the non-uniqueness will not influence any specific calculation that follows in this section.

\subsection{Stiefel Manifold \texorpdfstring{$\mathcal{V}_r(N) $}{Vr(N)} (Special Cases: \texorpdfstring{$\mathbb{S}^{N-1} $}{S(N-1)}, \texorpdfstring{$\SO(N)$}{SO(N)} and \texorpdfstring{$\Or(N)$}{O(N)}) }

We refer the readers to \cite{absil2008optimization,edelman1998geometry} for a detailed discussion on the definition and geometry of the Stiefel manifold. It is known that $\Or(N)$ acts transitively on $\mathcal{V}_r(N)$ and the isometry corresponding to each $O\in\Or(N)$ is $X\mapsto O.X:=OX$ for $X\in \mathcal{V}_r(N)$. The tangent space of $\Or(N)$ at identity is $\mathfrak{o}(N)=\mathscr{A}(\mathbb{R}^{N\times N})$, the space of all skew-symmetric $N\times N$-matrices, where $\{\mathcal{E}_{i j}:1\leq i<j\leq N\}$ is an orthogonal basis as introduced in \S\ref{KSD-Homogeneous}. We take local curves $O_{ij}(t)$ at identity corresponding to each $\mathcal{E}_{i j}$, i.e., $\frac{d}{d t}O_{i j}(t)|_{t=0}=\mathcal{E}_{i j}$, then the killing field corresponding to each $\mathcal{E}_{i j}$ is $K^{i j}_X:=\frac{d}{d t}O_{i j}(t)|_{t=0}.X=\frac{d}{d t}O_{i j}(t)|_{t=0}X=\mathcal{E}_{i j}X$, i.e., the vector field $K^{i j}$ that assigns each $X$ with tangent vector $\mathcal{E}_{i j}X\in T_X\mathcal{V}_r(N)$. We plug this into \eqref{kp-general} so that the $k_p$ function on Stiefel manifolds equals
\[
\kappa_p(X,Y)= \kappa\cdot \sum_{i<j} K^{i j}_X\log(p \kappa)\cdot K^{i j}_Y\log(p \kappa) + \kappa\cdot\sum_{i<j}K^{i j}_Y K^{i j}_X\log \kappa.
\]
The first term can be written in a closed form:
\[
\begin{aligned}
\kappa\cdot\sum_{i<j} K^{i j}_X\log(p \kappa)\cdot K^{i j}_Y\log(p \kappa) &=\kappa\cdot\sum_{i<j}\tr[X^\top\mathcal{E}^\top_{i j} \nabla^\mathbb{R}_X\log(p \kappa) ]\cdot\tr[Y^\top\mathcal{E}^\top_{i j} \nabla^\mathbb{R}_Y\log(p \kappa) ]
\\
&=\kappa\cdot\sum_{i<j} \langle \mathcal{E}_{i j}, \nabla^\mathbb{R}_X \log(p \kappa) X^\top\rangle_{\F} \cdot \langle \mathcal{E}_{i j}, \nabla^\mathbb{R}_Y \log (p\kappa) Y^\top\rangle_{\F} \\ \text{Based on Prop. \ref{summation} }\quad\quad
&= \kappa\cdot\langle\mathscr{A}(\nabla^{\mathbb{R}}_X \log (p \kappa) X^\top), \mathscr{A}(\nabla^\mathbb{R}_Y \log (p \kappa) Y^\top) \rangle_{\F}.
\end{aligned}
\]
The summation of $K^{i j}_Y K^{i j}_X\log \kappa$ relies on the specific form of $\kappa$ and thus has no closed form in general, which can not be further simplified. However, if $\kappa$ is a radial kernel, i.e., 
\[
\kappa(X,Y)=\exp(-\psi(\Vert X-Y\Vert^2_{\F}))=\exp(-\psi(2r-2\tr[X^\top Y])),\quad X,Y\in\mathcal{V}_r(N),
\]
for some $\psi\in C^2[0,+\infty)$, 
then $K^{i j}_Y K^{i j}_X\log \kappa$ equals
\[
\begin{aligned}
&K^{i j}_Y K^{i j}_X\log \kappa=-\frac{d^2}{d t dt'}\psi(2r-2\tr[X^\top O_{i j}^\top(t) O_{i j}(t')Y])|_{t,t'=0}\\
= &\underbrace{2 \psi'(\Vert X-Y\Vert^2_{\F})\cdot \tr[X^\top \mathcal{E}_{i j}^\top  \mathcal{E}_{i j}Y]}_{(i)}  
\underbrace{- 4\psi''(\Vert X-Y\Vert^2_{\F})\cdot \tr[X^\top\mathcal{E}^\top_{i j} Y]\cdot \tr[X^\top\mathcal{E}_{i j} Y]}_{(ii)}. 
\end{aligned}
\]
Note that  $\mathcal{E}_{i j}^\top \mathcal{E}_{i j}= \frac{1}{2}E_{i i}+\frac{1}{2}E_{j j}$, thus $\sum_{i<j} \mathcal{E}_{i j}^\top \mathcal{E}_{i j}=\frac{N-1}{2} I$. Therefore, the summation of $(i)$ over $i<j$ equals $
(N-1) \psi'(\Vert X-Y\Vert^2_{\F})\langle X, Y\rangle_{\F}$. The summation of $(ii)$ follows the same mechanics as in the Prop. \ref{summation}, which equals $ 4\psi''(\Vert X-Y\Vert^2_{\F})\Vert\mathscr{A}(XY^\top)\Vert^2_{\F} $.
We summarize into following theorem:

\begin{theorem} Given density $p$ and kernel $\kappa$, the $\kappa_p$ function on $\mathcal{V}_r(N)$ is given by
\begin{equation}\label{kp-Stiefel}
\begin{aligned}
\kappa_p(X,Y)&=\langle\mathscr{A}(\nabla^{\mathbb{R}}_X \log (p \kappa) X^\top), \mathscr{A}(\nabla^\mathbb{R}_Y \log (p \kappa) Y^\top) \rangle_{\F}\cdot \kappa(X,Y)\\
&+\cdot\sum_{i<j} K^{i j}_Y K^{i j}_X\log \kappa\cdot \kappa(X,Y).
\end{aligned}
\end{equation}
Furthermore, if $\kappa(X,Y)=e^{-\psi(\Vert X-Y\Vert^2_{\F})}$ is a radial kernel for some $\psi\in C^2[0,+\infty)$, then 
\begin{equation}\label{kp-Stiefel-radial}
\begin{aligned}
\kappa_p(X,Y)&=\langle\mathscr{A}(\nabla^{\mathbb{R}}_X \log (p \kappa) X^\top), \mathscr{A}(\nabla^\mathbb{R}_Y \log (p \kappa) Y^\top) \rangle_{\F} \cdot \kappa(X,Y)\\
& + (N-1) \psi'(\Vert X-Y\Vert^2_{\F})\langle X, Y\rangle_{\F}\cdot \kappa(X,Y)\\
&+4\psi''(\Vert X-Y\Vert^2_{\F})\Vert\mathscr{A}(XY^\top)\Vert^2_{\F}\cdot \kappa(X,Y).
\end{aligned}
\end{equation}
\end{theorem}

Next we explicitly calculate the closed form of MKSDE for the exponential family on $\mathcal{V}_r(N)$. It suffices to calculate the matrix $Q(X,Y)=k\sum_{i<j} (K^{i j}_X\zeta_k\cdot K^{i j}_Y\zeta_l)_{k l}$ and the vector $b(X,Y)=\kappa\cdot\sum_{i<j}[K^{ij}_X\eta+K^{i j}_X\log\kappa]K^{i j}_Y\zeta
$ introduced in (\ref{Q&b}). For $Q(X,Y)$, we have
\[
\begin{aligned}
\sum_{i<j}K^{i j}_X\zeta_k&\cdot K^{i j}_Y\zeta_l= \sum_{i<j} \tr[(\mathcal{E}_{i j} X)^\top\nabla^\mathbb{R}_X\zeta_k]\cdot\tr[(\mathcal{E}_{i j} Y)^\top\nabla^\mathbb{R}_Y\zeta_l] \\
&= \sum_{i<j}\langle \mathcal{E}_{i j},\nabla^\mathbb{R}_X\zeta_k X^\top   \rangle_{\F}\cdot \langle \mathcal{E}_{i j},\nabla^\mathbb{R}_Y\zeta_l Y^\top   \rangle_{\F} = \langle \mathscr{A}(\nabla^\mathbb{R}_X\zeta_k X^\top), \mathscr{A}(\nabla^\mathbb{R}_Y\zeta_l Y^\top)\rangle_{\F},
\end{aligned}
\]
Similarly, for $b(X,Y)$, we have
\[
\sum_{i<j}K^{i j}_X\log(e^\eta\kappa) \cdot K^{i j}_Y\zeta_k 
=\langle \mathscr{A}(\nabla^\mathbb{R}_X\log(e^\eta\kappa)X^\top),\mathscr{A}(\nabla^\mathbb{R}_Y\zeta_k Y^\top)\rangle_{\F}.
\]
To sum up, we have

\begin{theorem}[MKSDE for exponential family] Given  $p_\theta\propto\exp(\theta^\top\zeta(X)+\eta(X))$, let $Q(X,Y)$ and $b(X,Y)$ be the matrix and vector given by
\begin{equation}
\begin{aligned}
Q(X,Y)_{kl}&=\langle \mathscr{A}(\nabla^\mathbb{R}_X\zeta_k X^\top), \mathscr{A}(\nabla^\mathbb{R}_Y\zeta_l Y^\top)\rangle_{\F}\cdot \kappa(X,Y),\\
b(X,Y)_k&=\langle \mathscr{A}(\nabla^\mathbb{R}_X\log(e^\eta\kappa)X^\top),\mathscr{A}(\nabla^\mathbb{R}_Y\zeta_k Y^\top)\rangle_{\F}\cdot \kappa(X,Y), 
\end{aligned}
\end{equation}
Then the MKSDE can be computed via \eqref{MKSDE-Form}.
\end{theorem}

Next, we calculate several examples for specific radial kernel $\kappa$ and density $p$ on Stiefel manifold. To obtain the KSD in \eqref{kp-Stiefel-radial}, it suffices to compute $\nabla^\mathbb{R}_X\log \kappa$ , $\nabla^\mathbb{R}_X\log p$ and $\psi'$ and $\psi''$. To obtain the MKSDE of exponential family $p(X)\propto\exp(\theta^\top \zeta(X)+\eta(X))$ in \eqref{MKSDE-Form}, we only need to compute $\nabla^\mathbb{R}_X\zeta_i $ and $\nabla^\mathbb{R}_X\eta $, to construct the  

\subsubsection*{Commonly-used distributions} Most of the widely-used distribution families on Stiefel manifold $\mathcal{V}_r(N)$ have intractable normalizing constant, including:
\begin{itemize}
\item Matrix Fisher (MF) family: $p(X;F)\propto \exp[\tr(F^\top X)]$ with parameter $F\in\mathbb{R}^{N\times r }$, which belongs to the exponential family since $\log p=\tr(F^\top X)$ is linear in $F$. We have $\nabla^\mathbb{R}_X\log p = F $ and $\nabla^\mathbb{R}_X\eta=0$. Note that $\zeta(X)=X$, thus the Euclidean gradient of its $(i,j)^{\text{th}}$ component $\zeta_{i j}=\tr[E_{i j}^\top X]$ is $\nabla^\mathbb{R}_X \zeta_{i j}= E_{i j}\in\mathbb{R}^{N\times r}$. 
    \item Matrix Bingham (MB) family: $p(X;A)\propto \exp[\tr(X^\top A X)]$ with parameter $A\in \mathbb{R}^{N\times N}$, which belongs to the exponential family since $\log p=\tr[(X X^\top)^\top A]$ is linear in $A$. We have $\nabla^\mathbb{R}_X\log p=(A+\mathscr{S}(A)) X$ and $\nabla^\mathbb{R}_X\eta=0$. Note that $\zeta(X)=X X^\top$, thus its $(i,j)^{\text{th}}$ component $\zeta_{i j}=\tr[E_{i j}^\top (X X^\top)]$. Here $E_{i j}\in\mathbb{R}^{N\times N}$. Therefore, $\nabla^\mathbb{R}_X\zeta_{i j}=(E_{i j}+E_{j i})X$. The MB family is not identifiable, since any $A_1$, $A_2$ such that $\mathscr{S}(A_1)=\mathscr{S}(A_2)$ will correspond to the same distribution.
\item Matrix Fisher-Bingham (MFB) family: $p(X;A,F)\propto \exp[\tr(X^\top A X+ F^\top X)]$ with parameter $(A,F)\in\mathbb{R}^{N\times (N+r)}$, combining $A$ and $F$ by row, which belongs to the exponential family since $\log p=\tr(X^\top A X+ F^\top X)$ is linear in $(A,F)$. We have $\nabla^\mathbb{R}_X\log p= (A+\mathscr{S}(A)) X + F $, $\nabla^\mathbb{R}_X \eta=0$. Note that $\zeta(X)=(XX^\top,X)$, thus $\nabla^\mathbb{R}_X \zeta_{i j}=(E^N_{i j}+E^N_{j i})X$ for $1\leq j\leq N$, and $\nabla^\mathbb{R}_X \zeta_{i j}=E^r_{i j}$ for $N+1\leq j\leq N+r$, where $E^N_{i j}\in\mathbb{R}^{N\times N}$ and $E^r_{i j}\in\mathbb{R}^{N\times r}$. The MFB family is also non-identifiable.
    \item Riemannian Gaussian(RG) family: $p(X)\propto \exp(-\frac{d^2(X,\Bar{X})}{2\sigma^2})$ with parameters $\Bar{X}\in\mathcal{V}_r(N)$, $\sigma>0$. The RG family is not a member of the exponential family. For Riemannian Gaussian family, $\log p(X)=-\frac{1}{2\sigma^2}d^2(X,\Bar{X})$.  The Riemannian gradient of $d^2$ function is $-2\Log_X(\Bar{X})$
\cite{pennec2017hessian}, the Riemannian gradient of $\log p(X)$ is $\sigma^{-2}\Log_X(\Bar{X})$. Here $\Log$ is the Riemannian logarithm on $\mathcal{V}_r(N)$, which can be computed numerically \cite{zimmermann2017matrix}. It is shown in \cite[\S 2.4.1]{edelman1998geometry}
that the Riemannian metric on $\mathcal{V}_r(N)$ is given by $\langle D_1,D_2\rangle_X= \tr\big(D_1^\top (I-\frac{1}{2} X X^\top) D_2\big)$, thus $\nabla^\mathbb{R}_X\log p=\sigma^{-2} (I-\frac{1}{2}X X^\top)\Log_X\Bar{X}$.
\end{itemize}
\subsubsection*{Commonly-used kernels} The most widely-used kernel on $\mathcal{V}_r(N)$ includes:
\begin{itemize}
    \item Gaussian kernel:
    \[
    \kappa(X,Y)=\exp\left( -\frac{\tau}{2}\Vert X-Y\Vert^2_{\F} \right)\propto \exp(\tau \tr(X^\top Y)),\quad X,Y\in \mathcal{V}_r(N).
    \]
    Note that $\nabla^\mathbb{R}_X \log \kappa= \tau Y $.
    \item Inverse quadratic kernel:
    \[
    \kappa(X,Y)=(\beta+\Vert X-Y\Vert^2_{\F})^{-\gamma},\quad X,Y\in\mathcal{V}_r(N).
    \]
    Note that $\nabla^\mathbb{R}_X \log \kappa=\frac{2\gamma}{\beta+\Vert X-Y\Vert^2_{\F}} Y$.
\end{itemize}

\subsection{Grassmann Manifold \texorpdfstring{$\mathcal{G}_r(N)$}{Gr(N)} }

 We refer the readers to \cite{edelman1998geometry} for the detailed definition and geometry of Grassmann manifold. It is known that $\Or(N)$ acts transitively on $\mathcal{G}_r(N)$, and the isometry corresponding to each $O\in \Or(N)$ is $X\mapsto O.X= O X O^\top$. The tangent space of $\Or(N)$ at identity is $\mathfrak{o}(N)=\mathscr{A}(\mathbb{R}^{N\times N})$, the space of all skew-symmetric $N\times N$-matrices, where $\{\mathcal{E}_{i j}:1\leq i<j\leq N\}$ is an orthogonal basis as introduced in \S\ref{KSD-Homogeneous}. We take local curves $O_{ij}(t)$ at identity corresponding to each $\mathcal{E}_{i j}$, i.e., $\frac{d}{d t}O_{i j}(t)|_{t=0}=\mathcal{E}_{i j}$, then the killing field corresponding to each $\mathcal{E}_{i j}$ is $K^{i j}_X:=\frac{d}{d t}O_{i j}(t)|_{t=0}.X=\frac{d}{d t}O_{i j}(t)X O^\top_{i j}(t)|_{t=0}=\mathcal{E}_{i j} X-X\mathcal{E}_{i j}$, i.e., the vector field $K^{i j}$ that assigns each $X$ with tangent vector $\mathcal{E}_{i j} X-X\mathcal{E}_{i j}\in T_X\mathcal{G}_r(N)$. We plug this into \eqref{kp-general} so that the $\kappa_p$ function on $\mathcal{G}_r(N)$ equals
\[
\kappa_p(X,Y)=\kappa\cdot \sum_{i<j} K^{i j}_X\log(p \kappa)\cdot K^{i j}_Y\log(p \kappa) + \kappa\cdot\sum_{i<j}K^{i j}_Y K^{i j}_X\log \kappa.
\]
Then the first term can be written in a closed form:
\[
\begin{aligned}
&\kappa\cdot \sum_{i<j} K^{i j}_X\log(p \kappa)\cdot K^{i j}_Y\log(p \kappa)\\
&=\kappa\cdot \sum_{i<j} \tr[(\mathcal{E}_{i j} X-X\mathcal{E}_{i j})\nabla^\mathbb{R}_X\log(p \kappa)]\cdot \tr[(\mathcal{E}_{i j} Y-Y\mathcal{E}_{i j})\nabla^\mathbb{R}_Y\log(p \kappa)]\\
&= 4 \kappa\cdot\sum_{i<j} \langle \mathcal{E}_{i j}, \mathscr{S}(\nabla^\mathbb{R}_X\log(p \kappa)) X\rangle_{\F}\cdot\langle \mathcal{E}_{i j}, \mathscr{S}(\nabla^\mathbb{R}_Y\log(p \kappa)) Y\rangle_{\F}\\
\text{Based on Prop. \ref{summation}}\ 
&= 4 \kappa \langle\mathscr{A}[\mathscr{S}(\nabla^\mathbb{R}_X\log(p \kappa)) X], \mathscr{A}[\mathscr{S}(\nabla^\mathbb{R}_Y\log(p \kappa)) Y]\rangle_{\F}.
\end{aligned}
\]
The summation of $K^{i j}_Y K^{i j}_X\log k$ relies on the specific form of $\kappa$ and thus has no closed form in general, which can not be further simplified. However, if $\kappa$ is a radial kernel, i.e., 
\[
\kappa(X,Y)=\exp(-\psi(\Vert X-Y\Vert^2_{\F}))=\exp(-\psi(2r-2\tr[XY])),\quad X,Y\in\mathcal{G}_r(N),
\]
for some $\psi\in C^2[0,+\infty)$. Then we have
\[
\begin{aligned}
K^{i j}_Y K^{i j}_X\log \kappa &=-\frac{d^2}{d t dt'}\psi(2r-2\tr[O_{ij}(t) X O_{i j}^\top(t) O_{i j}(t')Y O_{i j}^\top(t')])|_{t,t'=0}\\
&= 2 \psi'(\Vert X-Y\Vert^2_{\F})\cdot\underbrace{\tr[( \mathcal{E}_{i j} X-X \mathcal{E}_{i j})(\mathcal{E}_{i j}Y-Y\mathcal{E}_{i j})]}_{(I)}\\
&- 4\psi''(\Vert X-Y\Vert^2_{\F})\cdot\underbrace{ \langle YX-XY,\mathcal{E}_{i j} \rangle_{\F}\cdot \langle XY-YX,\mathcal{E}_{i j} \rangle_{\F}}_{(II)}.
\end{aligned}
\]
Due to Prop. \ref{summation}, the summation of $(II)$ over $i<j$ equals $-\Vert XY-YX\Vert^2_{\F}$, as we note that $XY-YX$ is skew-symmetric itself. For $(I)$, note that $(I)=\tr[2 X\mathcal{E}_{i j} Y\mathcal{E}_{i j}- (XY+YX)\mathcal{E}_{i j}\mathcal{E}_{i j}]$. 
We set $X=(x_{i j})$ and $Y=(y_{i j})$ and then we have
\[
\begin{aligned}
\sum_{i<j} 2\tr[ \mathcal{E}_{i j} X \mathcal{E}_{i j} Y ]
&= \sum_{i<j} \tr[ (E_{j i}-E_{i j})X (E_{j i}-E_{i j}) Y ]\\
&= \sum_{i<j} \tr[E_{i j}X E_{i j}Y+E_{j i}X E_{j i}Y-E_{i j}X E_{ j i}Y-E_{j i}X E_{i j}Y]\\
&=\sum_{i<j} (x_{j i} y_{j i}+x_{i j}y_{i j}-x_{jj} y_{ii}- x_{ii} y_{jj})\\
&= \sum_{i\neq j} (x_{ij}y_{i j}-x_{ii}y_{jj})=\sum_{i,j} (x_{ij}y_{i j}-x_{ii}y_{jj})\\
&=\tr(XY)-\tr(X)\tr(Y)=\langle X,Y\rangle_{\F}-r^2.
\end{aligned}
\]
Furthermore, we have $\sum_{i<j}\tr[(XY+YX)\mathcal{E}_{i j}\mathcal{E}_{i j}]=(N-1)\langle X,Y\rangle_{\F}$ since $\sum_{i<j}\mathcal{E}_{i j}\mathcal{E}_{i j}=-\frac{N-1}{2} I_N$.
Therefore, the summation of $(I)$ over $i<j$ equals $N\langle X,Y\rangle_{\F}-r^2$. We summarize into following theorem:

\begin{theorem} Given density $p$ and kernel $k$, the $k_p$ function is given by
\begin{equation}\label{kp-Grassm}
\begin{aligned}
\kappa_p(X,Y)&=4\langle\mathscr{A}(\nabla^\mathbb{R}_X\log (p\kappa) X),\mathscr{A}(\nabla^\mathbb{R}_Y\log (p\kappa) Y)\rangle_{\F} \kappa(X,Y)\\
&+\sum_{i<j} K^{ij}_Y K^{ij}_X \log \kappa\cdot \kappa(X,Y),
\end{aligned}
\end{equation}
where all Euclidean gradients are chosen to be symmetric. Furthermore, if $\kappa$ is a radial kernel, i.e., $\kappa(X,Y)=e^{-\psi(\Vert X-Y\Vert^2_{\F})}$ for some $\psi\in C^2[0,+\infty)$, then the $\kappa_p$ function equals
\begin{equation}\label{kp-Grassm-radial}
\begin{aligned}
    \kappa_p(X,Y)
    &=4\langle\mathscr{A}(\nabla^\mathbb{R}_X\log (p\kappa) X), \mathscr{A}(\nabla^\mathbb{R}_Y\log (p\kappa) Y)\rangle_{\F} \kappa(X,Y) \\
    &+ 2\psi'(\Vert X-Y\Vert^2_{\F})\cdot (N\langle X,Y\rangle_{\F}-r^2) \kappa(X,Y)\\
    &+4\psi''(\Vert X-Y\Vert^2_{\F})\Vert XY-YX\Vert^2_{\F} \kappa(X,Y).
\end{aligned}
\end{equation}
\end{theorem}
Next we explicitly calculate the closed form of MKSDE for the exponential family on $\mathcal{G}_r(N)$. It suffices to calculate the matrix $Q(X,Y)=\kappa\sum_{i<j} (K^{i j}_X\zeta_k\cdot K^{i j}_Y\zeta_l)_{k l}$ and the vector $b(X,Y)=\kappa\cdot\sum_{i<j}[K^{ij}_X\eta+K^{i j}_X\log\kappa]K^{i j}_Y\zeta
$ introduced in \eqref{Q&b}. For $Q(X,Y)$, we have
\[
\begin{aligned}
\sum_{i<j}K^{i j}_X\zeta_k\cdot K^{i j}_Y\zeta_l&= \sum_{i<j} \tr[(\mathcal{E}_{i j} X-X\mathcal{E}_{i j})\nabla^\mathbb{R}_X\zeta_k]\cdot\tr[(\mathcal{E}_{i j} Y-Y\mathcal{E}_{i j})\nabla^\mathbb{R}_Y\zeta_l] \\
&= \sum_{i<j}\langle \mathcal{E}_{i j},\mathscr{S}(\nabla^\mathbb{R}_X\zeta_k) X   \rangle_{\F}\cdot \langle \mathcal{E}_{i j},\mathscr{S}(\nabla^\mathbb{R}_Y\zeta_l) Y   \rangle_{\F}\\
\text{Prop. \ref{summation} }\ 
&= \langle \mathscr{A}[\mathscr{S}(\nabla^\mathbb{R}_X\zeta_k) X], \mathscr{A}[\mathscr{S}(\nabla^\mathbb{R}_Y\zeta_l) Y]\rangle_{\F},
\end{aligned}
\]
Similarly, for $b(X,Y)$, we have
\[
\sum_{i<j}K^{i j}_X\log(e^\eta\kappa) \cdot K^{i j}_Y\zeta_k 
=\langle \mathscr{A}[\mathscr{S}(\nabla^\mathbb{R}_X\log(e^\eta\kappa))X],\mathscr{A}[\mathscr{S}(\nabla^\mathbb{R}_Y\zeta_k) Y]\rangle_{\F}.
\]
To sum up, we have

\begin{theorem}[MKSDE for exponential family] Given $p_\theta\propto\exp(\theta^\top\zeta(X)+\eta(X))$, let $Q(X,Y)$ and $b(X,Y)$ be the matrix and vector given by
\begin{equation}
\begin{aligned}
Q(X,Y)_{k l} &= \kappa(X,Y)\cdot \langle \mathscr{A}[\mathscr{S}(\nabla^\mathbb{R}_X\zeta_k) X], \mathscr{A}[\mathscr{S}(\nabla^\mathbb{R}_Y\zeta_l) Y]\rangle_{\F},\\
b(X,Y)_k &=\kappa(X,Y)\cdot\langle \mathscr{A}[\mathscr{S}(\nabla^\mathbb{R}_X\log(e^\eta\kappa))X],\mathscr{A}[\mathscr{S}(\nabla^\mathbb{R}_Y\zeta_k) Y]\rangle_{\F}, 
\end{aligned}
\end{equation}
Then the MKSDE can be computed via \eqref{MKSDE-Form}.
\end{theorem}

Next, we calculate several examples for specific kernel and families on $\mathcal{G}_r(N)$.

\subsubsection*{Commonly-used distributions} Most of the widely-used distribution families on Grassmann manifold have intractable normalizing constant, including:

\begin{itemize}
  \item Matrix Fisher (MF) family: $p(X)\propto \exp[\tr(F^\top X)]$ with parameter $F\in\mathbb{R}^{N\times N}$, which is a member of the exponential family since $\tr(F^\top X)$ is linear in $F$. We have $\nabla^\mathbb{R}_X\log p = \mathscr{S}(F) $ and $\nabla^\mathbb{R}_X\eta=0$. Note that $\zeta(X)=X$, thus the Euclidean gradient of its $(i,j)^{\text{th}}$ component $\zeta_{i j}=\tr[E_{i j}^\top X]$ is $\nabla^\mathbb{R}_X\zeta_{i j}=E_{i j}\in\mathbb{R}^{N\times N}$.
\item Riemannian Gaussian (RG) family: $p(X)\propto \exp(-\frac{d^2(X,\Bar{X})}{2\sigma^2})$ with parameters $\Bar{X}\in\mathcal{G}_r(N)$, $\sigma>0$. The RG family is not a member of the exponential family. For Riemannian Gaussian family, $\log p(X)=-\frac{1}{2\sigma^2}d^2(X,\Bar{X})$. The Riemannian gradient of $d^2$ function is $-2\Log_X(\Bar{X})$\cite{pennec2017hessian}, the Riemannian gradient of $\log p(X)$ is $\sigma^{-2}\Log_X(\Bar{X})$.  Here $\Log$ is the Riemannian logarithm on $\mathcal{G}_r(N)$, which can be computed numerically \cite{zimmermann2017matrix}.  It is known \cite[\S 2.4]{bendokat2020grassmann} that the Riemannian metric on $\mathcal{G}_r(N)$ is given by $\langle D_1,D_2\rangle_X= \frac{1}{2}\tr(D_1 D_2)$, thus $\nabla^\mathbb{R}_X\log p(X)=\frac{1}{2}\nabla^M_X\log p(X)=\frac{1}{2\sigma^2}\Log_X(\Bar{X})$. The Riemannian logarithm $\Log$ can be computed numerically \cite[\S 5.2]{bendokat2020grassmann}.
    \item Matrix Bingham (MB) and Matrix Fisher-Bingham (MFB) family: The MB and MFB family will degenerate to the matrix Fisher family on $\mathcal{G}_r(N)$, as $p(X)\propto\exp[\tr(X A X+F^\top X)]=\exp[\tr(AXX+FX)]=\exp[\tr((A+F)X)]$.
\end{itemize}

\subsubsection*{Commonly-used kernels} The most widely-used kernels on $\mathcal{G}_r(N)$ include
\begin{itemize}
    \item Gaussian kernel:
    \[
    k(X,Y) = \exp\left(-\frac{\tau}{2}\Vert X-Y\Vert^2_{\F}\right)\propto \exp(\tau\tr(X Y)),\quad X,Y\in\mathcal{G}_r(N).
    \]
    Note that $\nabla^\mathbb{R}_X\log k=\tau Y$.
    \item Inverse quadratic kernel:
    \[
    k(X,Y)=(\beta+\Vert X-Y\Vert^2_{\F})^{-\gamma}, \quad X,Y\in\mathcal{G}_r(N).
    \]
    Note that $\nabla^\mathbb{R}_X\log k=\frac{2\gamma}{\beta+\Vert X-Y\Vert^2_{\F}} Y$.
\end{itemize}

\subsection{Space of Symmetric Positive Definite Matrices
\texorpdfstring{ $\mathcal{P}(N)$}{P(N)}
 }

 It is known that $\GL(N)$, the general linear group of all non-singular matrices, acts transitively on $\mathcal{P}(N)$, where the isometry corresponding to each $G\in\GL(N)$ is $X\mapsto G.X=G^\top X G$. The tangent space of $\GL(N)$ at identity is $\mathfrak{gl}(N):=\mathbb{R}^{N\times N}$, where $\{E_{i j}:1\leq i,j\leq N\}$ is an orthogonal basis, as introduced in \S\ref{KSD-Homogeneous}. We take local curves $G_{i j}(t)$ of $E_{i j}$ on $\mathcal{P}(N)$ such that $\frac{d}{d t} G_{i j}(t)|_{t=0}=E_{i j}$, then the killing field corresponding to $E_{i j}$ is $K^{i j}_X:=\frac{d}{d t}G(t).X|_{t=0}=\frac{d}{dt} G_{i j}(t)^\top X G_{i j}(t)|_{t=0}=E_{j i} X+X E_{i j}$. We plug this into \eqref{kp-general} so that the $\kappa_p$ function on $\mathcal{P}(N)$ equals
\[
\kappa_p(X,Y)= \kappa\cdot \sum_{i<j} K^{i j}_X\log(p \kappa)\cdot K^{i j}_Y\log(p \kappa) + \kappa\cdot\sum_{i<j}K^{i j}_Y K^{i j}_X\log \kappa.
\]
The the first term can be written in a closed form:
\[
\begin{aligned}
&\kappa\cdot \sum_{i<j} K^{i j}_X\log(p \kappa)\cdot K^{i j}_Y\log(p \kappa)\\
&=\kappa\cdot \sum_{i<j} \tr[(E_{j i} X+X E_{i j})\nabla^\mathbb{R}_X\log(p \kappa)]\cdot \tr[(E_{j i} Y+Y E_{i j})\nabla^\mathbb{R}_Y\log(p \kappa)]\\
&= 4\kappa\cdot \sum_{i<j} \langle E_{i j}, X\mathscr{S}(\nabla^\mathbb{R}_X\log(p \kappa))\rangle_{\tr}\cdot \langle E_{i j}, Y\mathscr{S}(\nabla^\mathbb{R}_Y\log(p \kappa))\rangle_{\tr}\\
&= 4 \kappa \langle X\mathscr{S}(\nabla^\mathbb{R}_X\log(p \kappa)), Y \mathscr{S}(\nabla^\mathbb{R}_Y\log(p \kappa))\rangle_{\F}.
\end{aligned}
\]
The summation of $K^{i j}_Y K^{i j}_X \log \kappa$ relies on the specific form of $\kappa$ and thus can not be further simplified. However, if $\kappa$ is a radial kernel, i.e., 
\[
\kappa(X,Y)=\exp(-\psi(\Vert X-Y\Vert^2_{\F})),\quad X,Y\in\mathcal{P}(N),
\]
for some $\psi\in C^2[0,+\infty)$, then we have
\[
\begin{aligned}
K^{i j}_Y K^{i j}_X \log \kappa
&= 
2\psi'(\Vert X-Y\Vert^2_{\F}) \tr[(E_{j i} X+X E_{i j})(E_{j i} Y+Y E_{i j})]\\
&- 4\psi''(\Vert X-Y\Vert^2_{\F}) \tr[(E_{ji}X+X E_{ij})(X-Y)]\tr[(E_{ji}Y+Y E_{ij})(X-Y)]\\
&=4\psi'(\Vert X-Y\Vert^2_{\F}) \underbrace{\tr[X E_{i j} Y E_{i j}+X E_{i j} E_{j i} Y]}_{(i)}
\\
&+ 16\psi''(\Vert X-Y\Vert^2_{\F}) \underbrace{\tr[(X-Y)X E_{ij}]\tr[(X-Y)YE_{ij}]}_{(ii)}.
\end{aligned}
\]
Denote $X:=(x_{i j})$ and $Y=(y_{i j})$. Note that $\tr[E_{i j} X E_{i j} Y]= x_{j i} y_{i j} =x_{i j} y_{i j} $ since $X$ is symmetric. Also note that $E_{i j} E_{j i}= E_{i i}$. Therefore, the summation of $(i)$ over $i,j$ equals $4(N+1)\psi'(\Vert X-Y\Vert^2_{\F})\langle X,Y\rangle_{\F}$. The summations of $(ii)$ follows from the Prop. \ref{summation}, which equals $16\psi'(\Vert X-Y\Vert^2_{\F})\langle X(X-Y),Y(X-Y)\rangle_{\F}$. To sum up,
\begin{theorem}\label{SPD-kp} Given the kernel function $\kappa$, the $\kappa_p$ function is,
\begin{equation}
\begin{aligned}
\kappa_p(X,Y)&=4\langle X\mathscr{S}(\nabla^\mathbb{R}_X \log (p\kappa)), Y\mathscr{S}(\nabla^\mathbb{R}_Y \log (p\kappa)) \rangle_{\F} \cdot \kappa(X,Y) \\
&+ \sum_{i,j} K^{i j}_Y K^{i j}_Y\log \kappa\cdot \kappa(X,Y).
\end{aligned}    
\end{equation}
Furthermore, if $\kappa$ is a radial kernel, i.e., $\kappa(X,Y)=e^{-\psi(\Vert X-Y\Vert^2_{\F})}$ for some $\psi\in C^2[0,+\infty)$, 
\begin{equation}\label{SPD-kp-Form}
\begin{aligned}
\kappa_p(X,Y)
&= 4\langle X\mathscr{S}(\nabla^\mathbb{R}_X \log (p\kappa)), Y\mathscr{S}(\nabla^\mathbb{R}_Y \log (p\kappa)) \rangle_{\F} \cdot \kappa(X,Y)\\
&+4(N+1)\langle X,Y\rangle_{\F} \psi'(\Vert X-Y\Vert^2_{\F})\cdot \kappa(X,Y)\\
&+16\psi''(\Vert X-Y\Vert^2_{\F})\langle X(X-Y),Y(X-Y)\rangle_{\F}\cdot \kappa(X,Y).
\end{aligned}
\end{equation}
\end{theorem}
Next we explicitly calculate the closed form of MKSDE for the exponential family on $\mathcal{G}_r(N)$. It suffices to calculate the matrix $Q(X,Y)=\kappa\sum_{i,j} (K^{i j}_X\zeta_k\cdot K^{i j}_Y\zeta_l)_{k l}$ and the vector $b(X,Y)=\kappa\cdot\sum_{i,j}[K^{ij}_X\eta+K^{i j}_X\log\kappa]K^{i j}_Y\zeta
$ introduced in \eqref{Q&b}. For $Q(X,Y)$, we have
\[
\begin{aligned}
\sum_{i,j}K^{i j}_X\zeta_k\cdot K^{i j}_Y\zeta_l&= \sum_{i,j} \tr[E_{j i} X+X E_{i j})\nabla^\mathbb{R}_X\zeta_k]\cdot\tr[(E_{ji} Y-YE_{i j})\nabla^\mathbb{R}_Y\zeta_l] \\
&= \sum_{i,j}\langle E_{i j},X\mathscr{S}(\nabla^\mathbb{R}_X\zeta_k)    \rangle_{\F}\cdot \langle E_{i j},Y\mathscr{S}(\nabla^\mathbb{R}_Y\zeta_l)    \rangle_{\F}\\
\text{Prop. \ref{summation} }\ 
&= \langle X\mathscr{S}(\nabla^\mathbb{R}_X\zeta_k) , Y\mathscr{S}(\nabla^\mathbb{R}_Y\zeta_l)\rangle_{\F},
\end{aligned}
\]
Similarly, for $b(X,Y)$, we have
\[
\sum_{i<j}K^{i j}_X\log(e^\eta\kappa) \cdot K^{i j}_Y\zeta_k 
=\langle X\mathscr{S}(\nabla^\mathbb{R}_X\log(e^\eta\kappa)), Y\mathscr{S}(\nabla^\mathbb{R}_Y\zeta_k)\rangle_{\F}.
\]
To sum up, we have

\begin{theorem}[MKSDE for exponential family] Given $p_\theta\propto\exp(\theta^\top\zeta(X)+\eta(X))$, let $Q(X,Y)$ and $b(X,Y)$ be the matrix and vector given by
\begin{equation}
\begin{aligned}
Q(X,Y)_{kl}&=\kappa(X,Y)\cdot\langle X\mathscr{S}(\nabla^\mathbb{R}_X\zeta_k) , Y\mathscr{S}(\nabla^\mathbb{R}_Y\zeta_l)\rangle_{\F},\\
b(X,Y)_k&=\kappa(X,Y)\cdot\langle X\mathscr{S}(\nabla^\mathbb{R}_X\log(e^\eta\kappa)), Y\mathscr{S}(\nabla^\mathbb{R}_Y\zeta_k)\rangle_{\F}, 
\end{aligned}
\end{equation}
Then the MKSDE can be computed using \eqref{MKSDE-Form}.
\end{theorem}

\subsubsection*{Commonly-used distributions}
\begin{itemize}
     \item Wishart family: $p(X;V,r)\propto |X|^{(r-N+1)/2}\exp\left( -\frac{1}{2} \tr[V^{-1} X]  \right) $, with parameter $V\in\mathcal{P}(N)$, $1\leq N\leq  r\in\mathbb{N}_+$. The Wishart family is not a member of the exponential family, since the domain $\mathcal{P}(N)$ of $V$ is not a vector space. Due to Jacobi's formula \cite[Thm. 8.1]{magnus2019matrix}, we have $\nabla^\mathbb{R}_X\log p=-\frac{1}{2} V^{-1} + \frac{r-N+1}{2} X^{-1}$.
     \item Riemannian Gaussian family: $p(X)\propto \exp\big(-\frac{d^2(X,\Bar{X})}{2\sigma^2}\big)$ with parameters $\Bar{X}\in\mathcal{P}(N)$, $\sigma>0$. The Riemannian Gaussian family is not a member of the exponential family. As shown in \cite{moakher2005differential}, on $\mathcal{P}(N)$, $d^2(X,\Bar{X})=\tr(\Log^2(\Bar{X}^{-1}X))$. By \cite[Prop. 2.1]{moakher2005differential}, we have 
\[
\frac{d}{d t} d^2(X(t),\Bar{X}) = \frac{d}{d t} \tr(\Log^2(\Bar{X}^{-1}X(t)))= 2\tr\left[\Log(\Bar{X}^{-1} X(t)) X^{-1}(t) \frac{d}{d t}X(t) \right].
\]
Therefore, $\nabla^\mathbb{R}_X \log p(X)= -\sigma^{-2} X^{-1}\Log(\Bar{X}^{-1}X) $. 
 \end{itemize}
 
\subsubsection*{Commonly-used kernels} The most widely-used kernels on $\mathcal{P}(N)$ include
\begin{itemize}
\item Gaussian kernel:
\[
\kappa(X,Y)= \exp\left( -\frac{\tau}{2}\Vert X-Y \Vert_{\F}^2 \right),\quad X,Y\in\mathcal{P}(N).
\]
Note that $\nabla^\mathbb{R}_X\log\kappa=  \tau (Y-X) $.
\item Inverse quadratic kernel:
   \[
    \kappa(X,Y)=(\beta+\Vert X-Y\Vert^2_{\F})^{-\gamma}, \quad X,Y\in\mathcal{G}_r(N).
    \]
    Note that $\nabla^\mathbb{R}_X\log \kappa=\frac{2\gamma (Y-X)}{\beta+\Vert X-Y\Vert^2_{\F}}$.
\end{itemize} 

\section{Experiments}\label{experiments}

In this section, we present two experiments to demonstrate the power of our kernel Stein method on manifolds. In the first experiments, we will compare the MLE and MKSDE in estimating the parameters of the matrix Fisher distribution on a Stiefel manifold $\mathcal{V}_r(N)$, illustrating the advantage of MKSDE over MLE due to the presence of the intractable normalizing constant in the MLE. In the second experiment, we validate the power of composite goodness of fit test using MKSDE by comparing the matrix Fisher distribution and matrix Bingham distribution on a Stiefel manifold. The kernel we choose in these experiments is the Gaussian kernel $\kappa(X,Y)=\exp(\tr(X^\top Y))$, $X,Y\in\mathcal{V}_r(N)$. Code for computing the KSD, MKSDE and conducting the composite goodness of fit test is provided on GitHub at \url{https://github.com/cvgmi/KSD-on-Riemannian-Manifolds}.
\subsection{MKSDE vs. MLE}\label{MKSDEvsMLE}

The normalizing constant $c(F)$ of the matrix Fisher distribution $p(X)\propto \exp(\tr(F^\top X))$ on a Stiefel manifold is an intractable hypergeometric function of $F$. The most widely-used classical method to compute the MLE of the matrix Fisher distribution on a Stiefel manifold utilize two direct approximate solutions, introduced in \cite[\S 13.2.3]{mardia2000directional}. In general, the first solution approximates the MLE well when $F$ is small, while the second approximates the MLE relatively accurately when $F$ is large. However, it should be noted that both approximate solutions work poorly for medium-valued $F$, and the second approximate solution involves solving a non-linear multivariate equation, which is burdened with relatively hight computational cost.

In this experiment, we obtain the samples from a matrix Fisher distribution $p(X)\propto \exp(\tr(F_0^\top X))$ with ground truth $F_0$ on  $\mathcal{V}_r(N)$, then compute the MLE, $\hat{F}_{\text{MLE}}$, the MKSDE $\hat{F}_{\text{MKSDE-U}}$ and  $\hat{F}_{\text{MKSDE-V}}$ obtained by minimizing $U^w_n$ and $V^w_n$ in \eqref{Uw&Vw} respectively. The figure \ref{MKSDEvsMLE-Fig} shows the Frobenius distance between the ground truth $F_0$ and estimators including $\hat{F}_{\text{MLE}}$, $\hat{F}_{\text{MKSDE-U}}$ and $\hat{F}_{\text{MKSDE-V}}$ with varying values of $F_0$. Here we set $E_1=(1,0;1,0;1,0)\in\mathbb{R}^{3\times 2}$ and $E_2=(1,1;1,1;1,1)\in\mathbb{R}^{3\times 2}$ and the value of $F_0$ will vary in $0.3*E_1$, $E_1$, $5 * E_1$, $0.3 * E_2$, $E_2$ and $5 * E_2$.
As depicted in figure \ref{MKSDEvsMLE-Fig}, the approximate solution using the MLE worsens as $F_0$ becomes larger. This experiment demonstrates the performance of MKSDE which is independent of the normalization constant. 

\begin{figure*}[!t]
\centering
\subfloat[\(F_0 = 0.3*E_1\)]{\includegraphics[width =0.3\linewidth]{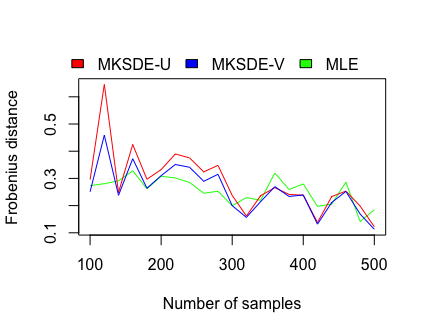}}
\hfil
\subfloat[$F_0 = E_1 $]{\includegraphics[width =0.3\linewidth]{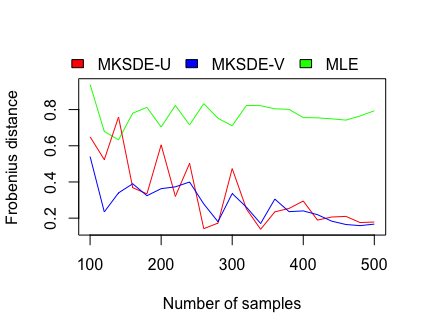}}
\hfil
\subfloat[\(F_0 = 5*E_1\)]{\includegraphics[width =0.3\linewidth]
{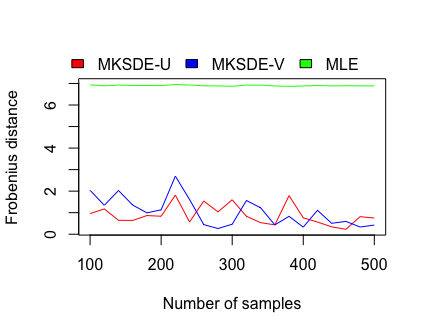}}
\hfil
\subfloat[\(F_0 = 0.3 * E_2\)]{\includegraphics[width =0.3\linewidth]
{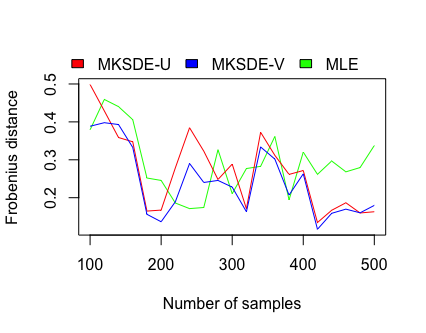}}
\hfil
\subfloat[\(F_0= E_2\)]{\includegraphics[width =0.3\linewidth]
{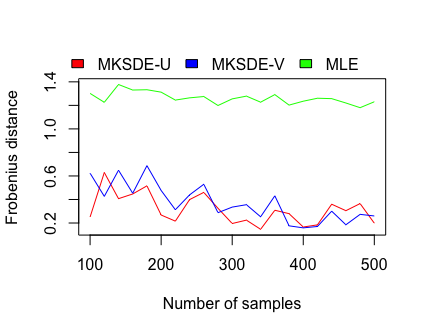}}
\hfil
\subfloat[\(F_0 = 5 * E_2 \)]{\includegraphics[width =0.3\linewidth]
{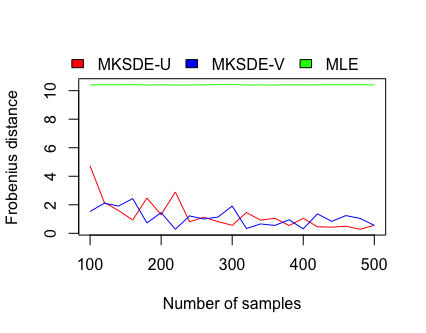}}
\caption{Frobenius distances between the estimators and ground truth }
\label{MKSDEvsMLE-Fig}
\end{figure*}
\vspace*{-10pt}

\subsection{Composite goodness of fit test}

In this experiment, we conduct the composite goodness of fit test presented in algorithm block \ref{GoF-Algo} to check whether a group of samples from a specific matrix Fisher distribution can be modeled by the matrix Bingham family. The table \ref{GoF-p-values} depicts the $p$ values of the composite goodness-of-fit test under different values of $F$ and the number of samples $n$. Intuitively, an MF distribution with small $F$ becomes nearly uniform, and an MF distribution with large $F$ will concentrate around the dominant directions specified by $F$. Therefore, the MF distributions belong approximately to the MB family when $A$ is small or large, but differ from the shape of the MB family for $F$ in-between. This is consistent with the results in Table \ref{GoF-p-values}. In addition, the loss function corresponding to the $V$-statistic shows better stability in the context of optimization compared to the $U$-statistic, this is because the global minimizer of $V^w_n$ always exists, as stated in Thm. \ref{MKSDE-Form-Thm}.

\begin{table}[ht]
\centering
\caption{ $p$-values of the composite goodness of fit test}
\label{GoF-p-values}
\subfloat[$p$-values of the $U$-stat]{
\begin{tabular}{lccccc}
\hline
number of samples & 100      & 150      & 200       & 250     & 300 \\
\hline

$F = 0.3 * E_1$     & 0.4670 & 0.3307 & 0.0582 & 0.0223 & 0.0034 \\
$F = E_1$     & 0.3420  & 0.0713 & 0.0320 & 0.0018 & 0.0007 \\
$ F = 5 * E_1$     & 0.2624 & 0.1528  & 0.0139 & 0.0282 & 0.0214\\\hline
\end{tabular}}

\subfloat[$p$-values of the $V$-stat]{
\begin{tabular}{lccccc}
\hline
number of samples & 100      & 150      & 200       & 250     & 300 \\
\hline

$F = 0.3 * E_1$     & 0.3923 & 0.1506 & 0.0348 & 0.0213 & 0.0028 \\
$F = E_1$     & 0.0687  & 0.0045 & 0.0012 & 0.0001 & 0.0000 \\
$ F = 5 * E_1$     & 0.0202 & 0.0173  & 0.0008 & 0.0030 & 0.0024\\\hline
\end{tabular}}
\end{table}

\section*{Acknowledgements}
This research was in part funded by the NIH NINDS and NIA grant RO1NS121099 to Vemuri.

\bibliography{reference.bib}
\bibliographystyle{plain}

\newpage
\appendix
\setcounter{page}{1}
\setcounter{section}{0}

\section{Proofs of Theorems (Original to this Work)}\label{proofs_supp}

\subsection{Proof of theorem \ref{Stein-identity}}\label{Proof-Stein-identity}

\begin{proof}
By Cauchy-Schwarz inequality, for $\vec{f}\in\mathcal{H}^m_\kappa$, we have
\[
\begin{aligned}
|\mathcal{T}_p \vec{f}(x)|&=|\langle \vec{f},(\vec{\mathcal{T}}_p \kappa)_x\rangle_{\mathcal{H}^m_\kappa}|\leq \Vert \vec{f}\Vert_{\mathcal{H}^m_\kappa}\cdot \Vert(\vec{\mathcal{T}}_p \kappa)_x\Vert_{\mathcal{H}^m_\kappa}=\Vert f\Vert_{\mathcal{H}^m_\kappa}\cdot \sqrt{\kappa_p(x,x)},\\
|\sum_{l=1}^m f_l(x) D_x^l|&\leq \sum_{l=1}^m |D^l_x| |\langle f_l,\kappa_x\rangle_{\mathcal{H}_\kappa}|\leq \sum_{l=1}^m |D^l_x| \Vert f_l\Vert_{\mathcal{H}_\kappa} \Vert \kappa_x\Vert_{\mathcal{H}_\kappa}\\
&\leq \sum_{l=1}^m |D^l_x| \Vert f_l\Vert_{\mathcal{H}_\kappa} \sqrt{\kappa(x,x)}\leq  \Vert f\Vert_{\mathcal{H}^m_\kappa} \sqrt{\kappa(x,x)}\sum_{l=1}^m |D^l_x|.    
\end{aligned}
\]
Therefore, $|\sum_{l=1}^m f_l D^l|$ and $\mathcal{T}_p \vec{f}$ are $P$-integrable for all $\vec{f}\in\mathcal{H}^m_\kappa$, since $\sqrt{\kappa(x,x)}\sum_{l=1}^m|D^l|$ and $\sqrt{\kappa_p(x,x)}$ are $P$-integrable. By Thm. \ref{DivThm}, $P(\mathcal{T}_p \vec{f})=\int_M \Div(\sum_{l=1}^m p f_l D^l) d\Omega = 0 $.
\end{proof}

\subsection{Proof of theorem \ref{vector-field-basis}}
\begin{proof}
 By Whitney embedding theorem \cite[Thm 6.15]{lee2013smooth}, there exists a smooth embedding $\psi:M\to \mathbb{R}^{2d+1}$. Let $D^l_x$, $1\leq l\leq 2d+1$, be the orthonormal projection of tangent vector $\big(\frac{\partial}{\partial x^l}\big)_x$ onto the tangent space $T_x \psi(M)$ of $\psi(M)$. Then $d\psi^{-1}(D^l)$, $l=1,\dots,2d+1$, is a group of vector fields on $M$ s.t. they span the entire tangent space of $M$ at each point.
 \end{proof}

\subsection{Proof of theorem \ref{KSD-Characterization-Compact} and \ref{KSD-Characterization-Noncompact}}

The proof of Thm. \ref{KSD-Characterization-Compact} and $\ref{KSD-Characterization-Noncompact}$ relies on results and techniques from \cite{simon2018kernel,barp2018riemann,qu2022framework}. To avoid redundancy, we omit detailed explanations here and refer the reader to these sources for a comprehensive background. 

\paragraph{Notation} Let $\mathcal{M}(M)=C_0(M)^*$ be the space of all signed finite Borel measures on $M$. For $\nu\in\mathcal{M}(M)$, $\int f d\nu$ is also written as $\nu(f)$. We can define the $\ksd$ between $P$ and $\nu$ similarly to \eqref{KSDdef-M} as, $\ksd(P,\nu):=\sup\{\nu(\mathcal{T}_p\vec{h}):\vec{h}\in\mathcal{H}^m_{\Tilde{\kappa}}\}$ and all results in \S\ref{KSDonM} apply. We also let $\mathcal{M}_{p,\psi}:=\{\nu\in\mathcal{M}(M): D^l\log p, |D^l|, |d\psi(D^l)|, 1\leq l\leq m \text{ are all } \nu\text{-integrable} \}$. Let 
$C^1_b(\mathbb{R}^d)$ be the space of all bounded continuously differentiable functions such that their gradients are also bounded. We equip $C^1_b(\mathbb{R}^d)$ with a special norm (introduced in \cite{simon2018kernel}) as follows: $f_n\to f$ in $C^1_b(\mathbb{R}^d)$ iff 
\begin{itemize}
    \item[(i)] $f_n\to f$, $|\nabla f_n-\nabla f|$ uniformly on any compact subset of $\mathbb{R}^d$;
    \item[(ii)] for any compact $A\subset\mathcal{M}(\mathbb{R}^d)$, $\sup_{\nu\in A}|\nu(f_n-f)|\to 0$,  $\sup_{\nu\in A}|\nu(\nabla f_n- \nabla f)|\to 0$.
\end{itemize}
Let $\mathscr{D}^1_{L^1}$ be the dual space of $C^1_b(\mathbb{R}^d)$. If $\kappa\in C^{(1,1)}(\mathbb{R}^d)$ is translation-invariant, then $\mathcal{H}_\kappa\hookrightarrow C^1_b(\mathbb{R}^d)$ (embeds into), thus for each $\Gamma\in \mathscr{D}^1_{L^1}$, there exists $\Phi_\Gamma\in\mathcal{H}_\kappa$ such that $\Gamma(f)=\langle f,\Phi_\Gamma\rangle_{\mathcal{H}_\kappa}$ for all $f\in\mathcal{H}_\kappa$. In fact, $\Phi_\Gamma(x)= \Gamma(\kappa_x) $. By \cite[Thm. 17]{simon2018kernel}, if $\kappa$ is further characteristic, then $\kappa$ is characteristic to $\mathscr{D}^1_{L^1}$, i.e., the map $\Phi:\mathscr{D}^1_{L^1}\to\mathcal{H}_\kappa$, $\Gamma\to\Phi_\Gamma$ is injective. Furthermore, let $[C^1_b(\mathbb{R}^d)]^m$ and $[\mathscr{D}^1_{L^1}]^m$ be the $m$-fold Cartesian products of $C^1_b(\mathbb{R}^d)$ and $\mathscr{D}^1_{L^1}$ respectively.

\begin{lemma}\label{P-ZeroDistribution}
$P\in\mathcal{P}_{p,\psi}\Longrightarrow P[\mathcal{T}_p(\vec{f}\circ\psi)]=0$ for all $\vec{f}\in [C^1_b(\mathbb{R}^{d'})]^m$.
\end{lemma}
\begin{proof} Given $\vec{f}\in [C^1_b(\mathbb{R}^{d'})]^m$, there exists $C>0$ such that $|f_l|\leq C$, $|\nabla f_l|\leq C $ for all $1\leq l\leq m$. Then $|p\cdot (f_l\circ \psi)\cdot D^l|\leq C p |D^l|$, $p|(f^l\circ\psi) D^l\log p|\leq Cp|D^l\log p|$, $p|D^l (f_l\circ\psi)|=p|[d\psi(D^l)]f_l|\leq p|d\psi(D^l)||\nabla f_l|\leq Cp|d\psi(D^l)|$ are all $\Omega$-integrable since 
$|D^l|$, $|D^l\log p|$, $|d\psi(D^l)|$ are all $P$-integrable. Then we have
\[
P[\mathcal{T}_p(\vec{f}\circ\psi)]=\int \sum_{l=1}^m [ (f_l\circ\psi) D^l\log p + D^l(f_l\circ\psi) ] p d\Omega = \int \sum_{l=1}^m\Div(p (f_l\circ\psi) D^l) d\Omega =0,
\]
by the generalized divergence theorem (Thm. \ref{DivThm}).
\end{proof}

\begin{lemma}\label{nu-zeroDistribution} For $\nu\in\mathcal{M}_{p,\psi}$, $\ksd(P,\nu)=0\Longrightarrow\nu[\mathcal{T}_p (\vec{f}\circ\psi)]=0$ for all $\vec{f}\in [C^1_b(\mathbb{R}^{d'})]^m$.
\end{lemma}
\begin{proof} By \cite[Prop. 7]{carmeli2010vector}, $\mathcal{H}_{\kappa_M}=\{f\circ\psi:f\in\mathcal{H}_\kappa \}$, thus $\ksd(P,\nu)=0\Longrightarrow\sup\{\nu(\mathcal{T}_p\vec{g}):\vec{g}\in\mathcal{H}^m_{\kappa_M}\}=0\Longrightarrow\sup\{\nu[\mathcal{T}_p(\vec{f}\circ\psi)]:\vec{f}\in\mathcal{H}^m_\kappa\}=0$, i.e., $\nu[\mathcal{T}_p(\vec{f}\circ\psi)]=0$ for all $\vec{f}\in\mathcal{H}^m_\kappa$. Let $|\nu|$ be the total variation of $\nu$. Since $|D^l\log p|$ and $|d\psi(D^l)|$ are all $\nu$-intergrable, we may define following two finite measures $\nu^l_1$ and $\nu^l_2$ as:
\[ \nu^l_1(A):=\int_A |D^l\log p| d|\nu|,\quad \nu^l_2(A):=\int_A |d\psi(D^l)| d|\nu|,\quad \text{for Borel } A\subset M.
\]
Note that $|\nu[\mathcal{T}_p(\vec{f}\circ\psi)]|\leq \sum_l \nu^l_1(|f_l|\circ\psi)+\sum_l \nu^l_2(|\nabla f_l|\circ \psi)\leq \sum_l \psi_*\nu^l_1(|f_l|)+\sum_l \psi_*\nu^l_2(|\nabla f_l|)$, where $\psi_*\nu^l_1$ and $\psi_*\nu^l_2$ are the pushforward measures of $\nu^l_1$ and $\nu^l_2$ by $\psi$.

Define the linear functional $s_\nu$ on $[C^1_b(\mathbb{R}^{d'})]^m$ as $s_\nu(\vec{f})= \nu[\mathcal{T}_p(\vec{f}\circ\psi)]$ for $\vec{f}\in [C^1_b(\mathbb{R}^{d'})]^m$. Given a converging sequence $\vec{f}_n\to \vec{f}$ in $[C^1_b(\mathbb{R}^{d'})]^m$  and using the definition of the topology on $[C^1_b(\mathbb{R}^{d'})]^m$, we have,
\[
|s_\nu(\vec{f}_n)-s_\nu(\vec{f})|\leq \sum_{l=1}^m \psi_*\nu^l_1(|f_{n,l}-f_l|)+\sum_{l=1}^m \psi_*\nu^l_2(|\nabla f_{n,l}-\nabla f_l|)\to 0,\quad \text{as } n\to\infty.
\]
Therefore, $s_\nu$ is continuous on $[C^1_b(\mathbb{R}^{d'})]^m$, thus $s_\nu\in[\mathscr{D}^1_{L^1}]^m$. By lemma \ref{P-ZeroDistribution} and the fact $\mathcal{H}^m_\kappa\subset [C^1_b(\mathbb{R}^{d'})]^m$, we have $s_P=0\in[\mathscr{D}^1_{L^1}]^m$ and $\Phi_{s_P}=0\in\mathcal{H}^m_\kappa$. The condition of the lemma also implies that $\Phi_{s_\nu}=0$. Since $\kappa$ is $\mathscr{D}^1_{L^1}$-characteristic, the embedding map $\Phi$ is injective, thus $\Phi_{s_\nu}=\Phi_{s_P}\Rightarrow s_\nu=s_P=0$.
\end{proof}

\subsubsection{Proof of theorem \ref{KSD-Characterization-Compact}}

\begin{proof}
We prove the forward direction $Q_n\Rightarrow P \Longrightarrow \ksd(P,Q_n)\to 0$ first. 

Since $p\in C^{s+1}$ and $M$ is compact, $\sqrt{\tilde{\kappa}(x,x)}|D^l|$ and $\tilde{\kappa}_p(x,y)$ are jointly continuous on $M\times M$, thus are bounded. As a result, the conditions in Thms. \ref{KSDFormThm} and \ref{Stein-identity} hold for all distributions on $M$, thus $\ksd(P,Q)=\iint \kappa_p(x,y) Q(dx) Q(dx)$ for all $Q\in\mathcal{P}(M)$ and $\iint \kappa_p(x,y) P(dx) P(dx)=0$. Moreover, we have $C_b(M)=C(M)$ as $M$ is compact. By Stone-Weiestrass theorem \cite[Thm. A.5.7]{steinwart2008support}, the space 
\[
C(M)\otimes C(M):=\left\{\sum_{k=1}^n f_k(x) g_k(y)\in C(M\times M) : f_k,g_k\in C(M), n\in\mathbb{N}_+\right\}
\]
is dense in $C(M\times M)$. As $Q_n(f)\to P(f)$ for all $f\in C(M)$ ($Q_n\Rightarrow P$), we have $(Q_n\times Q_n)(h)\to (P\times P)(h)$ for all $h\in C(M)\otimes C(M)$, which further implies $(Q_n\times Q_n)(h)\to (P\times P)(h)$ for all $h\in C(M\times M)$. Note that $\kappa_p(x,y)\in C(M\times M)$, thus 
\[
\ksd(P,Q_n)=\iint \kappa_p(x,y) Q(dx) Q(dy)\to \iint \kappa_p(x,y) P(dx) P(dy)=0.
\]

Next we prove the backward direction $\ksd(P,Q_n)\to 0\Longrightarrow Q_n\Rightarrow P$. 

Let $W^s_2(M)$ be the Sobolev space on $M$. If $\log p\in C^{s+1}(M)$, then for any $\eta\in C^\infty(M)\subset W^s_2(M)$ such that $P(\eta)=0$, their exists a $\zeta\in W^{s+2}_2(M)$ such that $\eta=\Delta \zeta+g(\nabla\log p,\nabla \zeta)=\frac{\Div(p\nabla \zeta)}{p}$ (see e.g., \cite{barp2018riemann}). The Sobolev embedding theorem states that $W^{s+2}_2(M)\subset C^2(M)$, thus $\zeta$ is $C^2$ and $\nabla \zeta$ is a $C^1$ vector field on $M$. Since the vector fields $D^l$, $1\leq l\leq D^l$ satisfies the assumption \ref{assumption-D}, there exists $h_1,\dots,h_m\in C^1(M)$ such that $\nabla \zeta=\sum_{l=1}^m h_l D^l $, thus $\eta=\frac{\Div(p\sum_{l=1} h_l D^l)}{p}=\mathcal{T}_p \vec{h}$. To sum up, for each $\bar{\eta}\in C^\infty(M)$ such that $P(\eta)=0$, there exists $\vec{h}\in [C^1(M)]^m$ such that $\mathcal{T}_p \vec{h}=\eta$.

The space $C^1_b(\mathbb{R}^{d'})$ restricted onto $M$ is apparently $C^1(M)=\{f\circ\psi: f\in C^1_b(\mathbb{R}^{d'})\}$. For $\nu\in\mathcal{M}_{p,\psi}=\mathcal{M}(M)$, by lemma $\ref{nu-zeroDistribution}$, $\ksd(P,\nu)=0\Longrightarrow \nu[\mathcal{T}_p(\vec{f}\circ\psi)]=0$ for all $\vec{f}\in[C^1_b(\mathbb{R}^{d'})]^m \Longrightarrow \nu[\mathcal{T}_p\vec{h}]=0 $ for all $\vec{h}\in [C^1(M)]^m\Longrightarrow \nu(\eta)=0$ for all $\eta\in C^\infty(M)$ s.t., $P(\eta)=0$. By Stone–Weierstrass theorem \cite[Thm. A.5.7]{steinwart2008support}, the space $\{\eta\in C^\infty(M):P(\eta)=0\}$ is dense in $P^\perp:=\{\eta\in C(M):P(\eta)=0\}$. Therefore, $\ksd(P,\nu)=0$ implies $\nu(\eta)=0$ for all $\nu\in P^\perp$, which further implies $\nu$ is proportional to $P$, as $\nu(\eta)=\nu[\eta-P(\eta)]+\nu[P(\eta)]=\nu(1)\cdot P(\eta)$ for all $\eta\in C(M)$.

Let $\mathcal{H}_{\kappa_p}:=\{ \mathcal{T}_p \vec{f}:\vec{f}\in \mathcal{H}^m_{\Tilde{\kappa}}\}$ be the range of the Stein pair, which is a subspace of $C(M)$. We just showed that, for any $\nu\in\mathcal{M}(M)$, if $\nu(\eta)=0$ for all $\eta\in\mathcal{H}_{\kappa_p}$, then $\nu$ is proportional $P$. A simple derivation by contradiction using the Hahn-Banach theorem yields that $\mathcal{H}_{\kappa_p}$ is dense in $P^\perp:=\{\eta\in C(M):P(\eta)=0\}$, as $\mathcal{M}(M)\simeq C(M)^*$ ($M$ is compact). Given a sequence of $Q_n\in\mathcal{P}(M)$, $\ksd(P,Q_n)\to 0$ implies that $Q_n(\eta)\to 0$ for all $\eta\in \mathcal{H}_{\kappa_p}$, which further implies $Q_n(\eta)\to 0$ for all $\eta\in P^\perp$. Therefore, for any $f\in C(M)$, $Q_n(f)=Q_n(f-P(f))+Q_n(P(f))=Q_n(f-P(f))+P(f)\to P(f)$ as $n\to \infty$, since $P(f)$ is a constant and $f-P(f)\in P^\perp$, thus $Q_n\Rightarrow P$.
\end{proof}

\subsubsection{Proof of theorem \ref{KSD-Characterization-Noncompact}}

\begin{proof}
It was proved in \cite[Thm. 3.4]{qu2022framework}, when $p>0$ is locally Lipshitcz continuous on $M$, the space $\{\Delta \zeta+g(\nabla\log p,\nabla \zeta): \zeta\in C^\infty_c(M)\}$ is dense in $L^2_0(P):=\{\eta\in L^2(P): P(\eta)=0\}$, the centered $L^2(P)$ space. For $\zeta\in C^\infty_c(M)$, $\nabla\zeta$ is a compactly supported smooth vector fields on $M$. Since $D^l$ satisfy assumption \ref{assumption-D}, there exists $\vec{h}\in [C^\infty_c(M)]^m$ such that $\nabla\zeta=\sum_{l=1}^m h_l D^l$, then $\Delta\zeta+g(\nabla\log p,\nabla\zeta)=\mathcal{T}_p \vec{h}$. Therefore, the space $\mathcal{H}^\infty_c :=\{\mathcal{T}_p \vec{h}:\vec{h}\in [C^\infty_c(M)]^m\}$ is dense in $L^2_0(P)$.

 Apparently, $C^\infty_c(M)\subset\{f\circ\psi:f\in C^1_b(\mathbb{R}^{d'})\}$. Given a $\nu\in \mathcal{M}_{p,\psi}\cap L^2(P)$, by lemma \ref{nu-zeroDistribution},  $\ksd(P,\nu)=0\Longrightarrow\nu[\mathcal{T}_p(\vec{f}\circ\psi)]=0$ for all $\vec{f}\in [C^1_b(\mathbb{R}^{d'})]^m\Longrightarrow \nu[\mathcal{T}_p\vec{h}]=0 $ for all $\vec{h}\in [C^\infty_c(M)]^m\Longrightarrow \nu(\eta)=0 $ for all $\eta\in\mathcal{H}^\infty_c\Longrightarrow P(\eta\cdot \frac{d\nu}{dP})=0$ for all $\eta\in \mathcal{H}^\infty_c\Longrightarrow P(\eta\cdot \frac{d\nu}{dP})=0$ for all $\eta\in L^2_0(P)\Longrightarrow \frac{d\nu}{dP}$ is a constant. 
 
For the first part of the conclusion of the theorem, $Q=P\Longrightarrow \ksd(P,Q)=0$ follows from the fact that $P\in\mathcal{P}_{p,\psi}$ and lemma \ref{P-ZeroDistribution}. For the reverse direction, we replace $\nu$ with $Q$ in above derivation, concluding that $\frac{dQ}{dP}\equiv1$, i.e., $Q=P$.

Next we prove the second part of the conclusion of the theorem. 

For the forward direction, we have $Q(\sqrt{\tilde{\kappa}(x,x)})=P(\sqrt{\tilde{\kappa}(x,x)}\cdot \frac{d Q_n}{d P})\leq P(\tilde{\kappa}(x,x))\cdot \Vert Q\Vert_P <+\infty$ for $Q\in L^2(P)$ by Hölder's inequality. Since translation-invariant kernels are bounded, $\tilde{\kappa}$ is also bounded, thus $\sqrt{\tilde{\kappa}(x,x)}|D^l|$ is integrable w.r.t. all $Q\in\mathcal{P}_{p,\psi}$. Therefore, the Conditions in Thm. \ref{KSDFormThm} and Thm. \ref{Stein-identity} hold for all $Q\in\mathcal{P}_{p,\psi}\cap L^2(P)$, thus we have $\ksd(P,Q)=\iint \kappa_p(x,y) Q(dx) Q(dx)$ for all $Q\in\mathcal{P}(M)$ and $\iint \kappa_p(x,y) P(dx) P(dx)=0$.

Let $Q_n\in\mathcal{P}_{p,\psi}\cap L^2(P)$ be a sequence of distributions such that $Q_n\Rightarrow P$. 
Given the condition $\sup_n\Vert Q_n\Vert_P<+\infty$ together with the fact that $C_b(M)$ is dense in $L^2(P)$, we conclude that  $Q_n(\eta)=P(\eta\cdot \frac{dQ_n}{dP})\to P(\eta)$ for all $\eta\in L^2(P)$. By a well-known result \cite[\S 5.5 Exer. 61]{folland1999real}, the space
\[
L^2(P)\otimes L^2(P)=\left\{ \sum_{k=1}^n f_k(x) g_k(y)\in L^2(P\times P): f_k,g_k\in L^2(P), n\in\mathbb{N}_+\right\}
\]
is dense in $L^2(P\times P)$. Since $(Q_n\times Q_n)(h)\to (P\times P)(h)$ for all $h\in L^2(P)\otimes L^2(P)$, thus $(Q_n\times Q_n)(h)\to (P\times P)(g)$ for all $h\in L^2(P\times P)$, as $\frac{dQ_n}{dP}(x)\cdot \frac{dQ_n}{d P}(y)\in L^2(P\times P)$. Note that $\kappa_p(x,y)\leq \sqrt{\kappa_p(x,x)}\sqrt{\cdot\kappa(y,y)}$ is in $L^2(P\times P)$, thus 
\[
\ksd(P,Q_n)=\iint \kappa_p(x,y) Q_n(dx) Q_n(dy)\to\iint \kappa_p(x,y) P(dx) P(dy) =0.
\]


For the reverse direction, Let $\mathcal{H}_{\kappa_p}:=\{ \mathcal{T}_p\vec{h}:h\in\mathcal{H}^m_{\Tilde{\kappa}} \}$ be the range of the Stein pair, which is a subspace of $ L^2_P$ since $P(\Tilde{\kappa}_p(x,x))<+\infty$. We have shown that $\nu(\eta)=0$ for all $\eta\in\mathcal{H}_{\kappa_p}\Longrightarrow \frac{d\nu}{dP}$ is a constant, which implies that $\mathcal{H}_{\kappa_p}$ is dense in $L^2_0(P)$. Therefore, $\ksd(P,Q_n)\to 0\Longrightarrow Q_n(\eta)=P(\eta\cdot\frac{dQ_n}{dP})\to 0=P(\eta)$ for all $\eta\in \mathcal{H}_{\kappa_p}$, which implies $Q_n(\eta)\to P(\eta)$ for all $\eta\in L^2_0(P)$, as $\mathcal{H}_{\kappa_p}$ is dense in $L^2_0(P)$ and $\sup_n\Vert \frac{dQ_n}{dP}\Vert_{L^2(P)}<+\infty$, which further implies that $Q_n(\eta)=Q_n(\eta-P(\eta))+Q_n(P(\eta))\to 0+P(\eta)=P(\eta)$ for all $\eta\in C_b(M)$. Therefore, $\ksd(P,Q_n)\Longrightarrow Q_n\Rightarrow P$.
\end{proof}

\subsection{Proof of theorem \ref{KSD-Characterization-weak}}

\begin{proof}
The forward direction is straightforward by Thm. \ref{Stein-identity}. It suffices to prove backward direction $Q=P\Longleftarrow\ksd(P,Q) $. Since $\kappa$ is $C_0$-universal, it is bounded by some constant $C>0$, so is $\Tilde{\kappa}$. Denote $\vec{D}\log(p/q):=(D^1\log(p/q),\dots,D^m\log(p/q))$, then we have
\[
\begin{aligned}
\sqrt{\kappa_q(x,x)}&=\Vert (\vec{\mathcal{T}}_q\kappa)_x\Vert^2=  \Vert (\vec{\mathcal{T}}_p\kappa)_x - \vec{D}\log(p/q)\cdot \kappa_x
\Vert^2 \leq \Vert (\vec{\mathcal{T}}_p\kappa)_x\Vert^2+\Vert \vec{D}\log(p/q)\cdot \kappa_x
\Vert^2\\
&=\sqrt{\kappa_p(x,x)}+\Vert \vec{D}\log(p/q)\Vert_{\mathbb{R}^m} \cdot\sqrt{\kappa(x,x)} \leq \sqrt{\kappa_p(x,x)}+\sqrt{C}\cdot \sum_{l=1}^m|D^l\log(p/q)|,
\end{aligned}
\]
which implies $\sqrt{\kappa_q(x,x)}$ is $Q$-integrable since $\sqrt{\kappa_p(x,x)}$ and $D^l\log(p/q)$ are $Q$-integrable, thus $Q(\mathcal{T}^l_q h)=0$ for all $h\in\mathcal{H}_{\Tilde{\kappa}}$ by Thm. \ref{Stein-identity}. The condition $\ksd(P,Q)=0$ further implies that $Q(\mathcal{T}^l_p h)=0$ for all $ h\in\mathcal{H}_{\Tilde{\kappa}}$, thus $Q(h\cdot D^l\log(p/q))=Q(\mathcal{T}^l_p h-\mathcal{T}^l_q h) =0$ for all $h\in\mathcal{H}_{\Tilde{\kappa}}$. Let $Q^w$ be the signed measure defined by $dQ^w:= D^l\log(p/q) dQ$, which is finite since $D^l\log(p/q)$ is $Q$-integrable. By \cite[Prop. 7]{carmeli2010vector}, $\mathcal{H}_{\Tilde{\kappa}}=\{f\circ\psi:f\in\mathcal{H}_\kappa\}$, thus $Q[(f\circ\psi)\cdot D^l\log(p/q)]=0$ for all $f\in\mathcal{H}_\kappa$, thus $\psi_*Q^w(f)=0$ for all $f\in\mathcal{H}_\kappa$, where $\psi_*Q^w$ is the pushforward measure of $Q^w$ by $\psi$. Since $\mathcal{H}_\kappa$ is dense in $C_0(\mathbb{R}^{d'})$, we have $\psi_*Q^w=0$ and further $D^l\log(p/q)=0$. Therefore, we conclude $p=q$ by assumption \ref{assumption-D} and connectedness of $M$.
\end{proof}

\subsection{Proof of Theorem \ref{sub-killing-field}}

\begin{proof}
It suffices to show that: for all $x\in H$ and $Y\in T_x H$, there exists $E\in T_e G$ s.t. its corresponding killing field $K$ satisfies $K_x=Y$. For each $x\in H$, let $\mathcal{C}_x$ be the stabilizer of $G$ at $x$, i.e., $\mathcal{C}_x:=\{g\in G:g.x=x\}$, which is a Lie subgroup of $G$. Since $T_e \mathcal{C}_x$ is a subspace of $T_e G$, we take a subspace $\mathscr{C}_x$ of $T_e G$ such that $T_e \mathcal{C}_x\oplus \mathscr{C}_x=T_e G $ and take a basis $D^j$, $1\leq j\leq \dim \mathscr{C}_x$, of $\mathscr{C}_x$. Note that this basis corresponds to a group of vector field $K^j$ on $H$, which are linearly independent at $x$. It is known that $H$ is diffeomorphic to $G/\mathcal{C}_x$ and $\dim H=\dim G-\dim \mathcal{C}_x $, thus $\dim H=\dim \mathscr{C}_x$. Thus $K^j_x$ span the entire $T_x H$.
\end{proof}

\subsection{Proof of theorems \ref{KSDasymptotic}, \ref{KSDasymptotic-stronger}, \ref{MKSDEconsistency}, \ref{MKSDECLT} and \ref{GoF}}

For Thm. \ref{KSDasymptotic}, see the proof of \cite[Thm. 4.5]{qu2023kernel}. 

For Thm. \ref{KSDasymptotic-stronger}, see \cite[\S B]{qu2023kernel}.

For Thm. \ref{MKSDEconsistency}, see \cite[\S C]{qu2023kernel}.

For Thm. \ref{MKSDECLT}, see the proof of \cite[Thm. 4.8]{qu2023kernel}.

For Thm. \ref{GoF}, see \cite[\S D]{qu2023kernel}.

\subsection{Proof of theorem \ref{MKSDE-Form-Thm}}

\begin{proof} Note that 
\[
Q(x,y)^\top=\kappa(x,y) \sum_{l=1}^m (K^l_x\zeta \cdot K^l_y\zeta^\top)^\top = \kappa(y,x)\sum_{l=1}^m K^l_y\zeta \cdot K^l_x\zeta^\top = Q(y,x),
\]
thus
\[
Q_u^\top = \sum_{i\neq j}\frac{Q(x_i,x_j)^\top}{n(n-1)}\frac{q(x_i)q(x_j)}{w(x_i)w(x_j)}=\sum_{j\neq i}\frac{Q(x_j,x_i)}{n(n-1)}\frac{q(x_j)q(x_i)}{w(x_j)w(x_i)} = Q_u,
\]
and $Q_v$ is also symmetric by a parallel argument. Next we show $Q_v$ is positive semi-definite.
Given arbitrary fixed vector $\theta_0\in\mathbb{R}^s$, we let $\phi^l_x = \theta_0^\top \cdot K^l_x\zeta$, which is a real-valued function on $H$. We let $\mathbf{1}_n:=(1,\dots,1)^\top\in\mathbb{R}^n$, $\Vec{\phi^l}:=(\phi^l_{x_1},\dots,\phi^l_{x_n})^\top\in\mathbb{R}^n$ and further set
\[
 \Phi=\sum_{l=1}^m \Vec{\phi^l}\Vec{\phi^l}^\top\in\mathbb{R}^{n\times n},\quad G_n:= \left(\kappa(x_i,x_j)\frac{q(x_i)q(x_j)}{w(x_i)w(x_j)}\right)_{i j }\in\mathbb{R}^{n\times n}.
\]
Note that $\Phi$ and $G_n$ are both positive semi-definite, thus by Schur product theorem, their Hadamard product (elementwise product) $G_n\circ \Phi$ is also positive semi-definite. Furthermore, since $\theta^\top  Q(x,y)\theta= \kappa(x,y)\sum_{l=1}^m \phi^l_x\cdot \phi^l_y$, we have
\[
\theta^\top  Q_v\theta= 
n^{-2}\mathbf{1}_n^\top (G_n\circ \Phi) \mathbf{1}_n \geq 0,
\]
which implies $Q_v$ is positive semi-definite by the arbitrariness of $\theta$. The rest of the theorem is straightforward.    
\end{proof}

\end{document}